\documentclass[12pt,a4paper]{article}
\usepackage[utf8]{inputenc}
\usepackage[T5]{fontenc}
\usepackage{amsmath, amssymb, amsthm, verbatim, hyperref, dsfont}
\usepackage{mathrsfs}
\usepackage[dvipsnames]{xcolor}
\usepackage{ragged2e}
\usepackage{marginnote}
\usepackage{enumerate}
\usepackage{makecell}
\usepackage{longtable}
\usepackage{epsfig}
\usepackage{tikz}
\usepackage{ytableau}
\usepackage{hyperref}
\usepackage{caption}
\usepackage{subcaption}
\hypersetup{hidelinks}
\usepackage{fancyvrb}
\usepackage{tikz-cd}
\usepackage[left=2.50cm, right=2.50cm, top=2.00cm, bottom=2.00cm]{geometry}
\usetikzlibrary {arrows.meta}
\usepackage{ifthen}
\def\LR{\operatorname{LR}}
\newcommand{\BK}{\operatorname{BK}}
\newcommand{\Hive}{\operatorname{Hive}}
\newcommand{\Skep}{\operatorname{Skep}}
\newcommand{\SkewHive}{\operatorname{SkewHive}}
\newcommand{\SkewSkep}{\operatorname{SkewSkep}}
\newcommand{\hive}{\operatorname{hive}}
\newcommand{\skep}{\operatorname{skep}}
\newcommand{\skewhive}{\operatorname{skewhive}}
\newcommand{\skewskep}{\operatorname{skewskep}}
\newcommand{\ttop}{\operatorname{top}}
\newcommand{\bbot}{\operatorname{bot}}
\newcommand{\hht}{\operatorname{ht}}
\newcommand{\Ht}{\operatorname{Ht}}

\definecolor{goodgreen}{rgb}{0.01, 0.75, 0.24}

\theoremstyle{plain}

\newtheorem{thm}{Theorem}[section]

\newtheorem{cor}[thm]{Corollary}
\newtheorem{lemma}[thm]{Lemma}
\newtheorem{conjecture}[thm]{Conjecture}

\theoremstyle{definition}
\newtheorem{definition}[thm]{Definition}
\newtheorem{example}[thm]{Example}
\theoremstyle{remark}

\newtheorem{remark}[thm]{Remark}
\newtheorem{bijection}[thm]{Bijection}

\newcommand{\ZZ}{\mathbb{Z}}

\newcommand{\RR}{\mathbb{R}}

\newcommand{\YY}{\mathbb{Y}}

\newcommand{\SSS}{\mathcal{S}}
\newcommand{\WW}{\mathcal{W}}

\newcommand{\Z}{\mathbb{Z}}
\newcommand{\ti}{\widetilde}
\newcommand{\sbs}{\mathbb{\subseteq}}
\usepackage{mathtools}
\DeclareMathOperator{\sort}{sort}
\DeclarePairedDelimiter\floor{\lfloor}{\rfloor}
\DeclarePairedDelimiter\abs{\lvert}{\rvert}
\DeclarePairedDelimiter\ceil{\lceil}{\rceil}

\title{\vspace{-1em} 
Skew Hives, Skew Skeps, Skew Schur Log-Concavity}

\author{\vspace{-1em} \\
\parbox{5cm}{\centering
Tuong Le\thanks{\href{mailto:tuongle@princeton.edu}{tuongle@princeton.edu}. Princeton University.}\\
{\small Lê Vũ Tường}
}%
\hspace{1em}%
\parbox{5cm}{\centering
Son Nguyen\thanks{\href{mailto:sonnvt@mit.edu}{sonnvt@mit.edu}. Massachusetts Institute of Technology.}\\
{\small Nguyễn Văn Thanh Sơn}
}
}

\date{\vspace{-2em}}
\allowdisplaybreaks

\begin{document}
\ytableausetup{centertableaux}

\maketitle

    \begin{abstract}
    Knutson and Tao's hives is a combinatorial model to compute Littlewood--Richardson coefficients. Similar to hives, Speyer introduced skeps and used them to prove a Schur log-concavity conjecture by Lam--Postnikov--Pylyavskyy. We first introduce skew hive and skew skep models, which specialize to both hives and skeps, and use this to prove a skew Schur log-concavity result generalizing Lam--Postnikov--Pylyavskyy conjecture. As a consequence, we obtain some log-concavity results concerning Newell--Littlewood numbers and shadow skew Schur functions. Finally, we explain bijections between (skew) hives, (skew) skeps, and peelable tableaux by Nguyen--Nguyen--Woodruff, answering Speyer's question.
    \end{abstract}

\tableofcontents


\section{Introduction}

\subsection{Schur log-concavity}

    Let $\lambda,\mu$ be two dominant weights of $\mathfrak{sl}_n$. Recall that the weight lattice is $\ZZ^n/(1,\ldots,1)$, so we view $\lambda,\mu$ as partitions whose $n$th part is zero. Let $\alpha_{ij} = e_i-e_j$ be the roots of the type A root system. An \textit{alcoved polytope} is the intersection of some half-spaces bounded by the hyperplanes $\{\tau~|~\langle \alpha_{ij},\tau\rangle = m\}$. See \cite{lam2007alcoved,lam2018alcoved} for a study of alcoved polytopes. Given two weights $\lambda,\mu$, the \textit{parallelepiped} $P_{\lambda,\mu}$ is the minimal alcoved polytope containing $\lambda$ and $\mu$. Lam--Postninkov--Pylyavskyy made a fundamental conjecture about when a tensor product $V_{\lambda}\otimes V_\mu$ is contained in $V_{\nu}\otimes V_\rho$ in terms of the parallelepiped. This conjecture can be phrased in terms of Schur positivity as follows.

    \begin{conjecture}[{\cite[Conjecture 15.1]{pylyavskyy2007comparing}}]\label{conj:log-concavity}
        If $\lambda + \mu = \nu + \rho$ and $\nu,\rho\in P_{\lambda,\mu}$, then $s_{\nu}s_\rho - s_{\lambda}s_\mu$ is Schur positive.
    \end{conjecture}

    This conjecture was open for 20 years, during which many special cases of were proved using different techniques \cite{king2008schur, purbhoo2008on, mcnamara2009positivity, ballantine2014schur, dobrovolska2007products, nguyen2025shuffle}. It was finally resolved recently by Speyer in \cite{speyer2026log}. The main tool of Speyer's proof is the \textit{skep}, which is a combinatorial model to compute the Littlewood--Richardson coefficients. Skeps can be obtained from hives, introduced by Knutson--Tao in \cite{knutson1999honeycomb}, via the octahedron recurrence. The key difference is that in skeps, traversing the lower and left boundary gives $(\lambda_1,\mu_1,\ldots,\lambda_n,\mu_n)$ while in hives, this gives $(\lambda_1,\ldots,\lambda_n,\mu_1,\ldots,\mu_n)$. This allows Speyer to apply $L$-convexity results to prove the conjecture.

    Our first goal is to generalize hives and skeps to skew hives and skew skeps, which can be used to compute the Littlewood--Richardson coefficients when multiplying two skew Schur functions. Then, using Speyer's technique, we obtain a skew Schur log-concavity analog of Conjecture \ref{conj:log-concavity}.

    \begin{thm}\label{thm:skew-schur-log}
        Let $\lambda,\mu,\nu,\rho, \ti\lambda,\ti\mu,\ti\nu,\ti\rho$ be partitions, not necessarily ending in $0$. 
        Denote $(\lambda,\mu) = (\lambda_1,\ldots,\lambda_n,\mu_1,\ldots,\mu_n)$. If $(\lambda,\mu) + (\nu,\rho) = (\ti\lambda,\ti\mu) + (\ti\nu,\ti\rho)$, and $(\ti\lambda,\ti\mu),(\ti\nu,\ti\rho)\in P_{(\lambda,\mu),(\nu,\rho)}$, then $s_{\ti\lambda/\ti\mu}s_{\ti\nu/\ti\rho} - s_{\lambda/\mu}s_{\nu/\rho}$ is Schur positive.
    \end{thm}

    Theorem~\ref{thm:skew-schur-log} simultaneously generalizes Conjecture \ref{conj:log-concavity} and various skew-schur log-concavity results proved by Lam--Postnikov--Pylyavskyy: Theorem 5, Theorem 12 and Corollary 14 of  {\cite{LPP07}}.

\subsection{Non-$L$-log-concavity of $c^\nu_{\lambda,(\pi-\lambda)}$}
    Speyer's proof techinique shows that for fixed $\nu$ and $\pi$ with $|\nu| = |\pi|$, the function $\lambda\mapsto c_{\lambda,(\pi-\lambda)}^{\nu}$ is a sum of $L$-log-concave functions. Since sums of $L$-log-concave functions need not be $L$-log-concave, Speyer asked \cite[Question 2.17]{speyer2026log} if this function is always $L$-log-concave. Unfortunately, the answer is no, as the following example shows.
    
    \begin{example}
        Let $\nu = (9, 5, 3), \pi = (6,5,4,2), x = (6,4,3), y = (5,3,1).$ Let $f(\lambda) := \lambda\mapsto c_{\lambda,(\pi-\lambda)}^{\nu}.$ Then $f(x)f(y) = 6 > 5 = f(\ceil*{\frac{x+y}2}) f(\floor*{\frac{x+y}2}).$
    \end{example}
    
\subsection{Newell--Littlewood numbers and Koike-Terada log-concavity}

    \begin{definition}[\cite{newell1951modification,littlewood1958products}]
        The Newell--Littlewood numbers are defined as
        \[ N_{\mu,\nu,\lambda} = \sum_{\alpha,\beta,\gamma}c_{\alpha,\beta}^\mu c_{\alpha,\gamma}^\nu c_{\beta,\gamma}^\lambda, \]
        where the indices are partitions with at most $n$ rows.
    \end{definition}

    Newell--Littlewood numbers are  tensor product multiplicities for type $B,C,D$ classical Lie groups. Specifically, let $W$ be a complex vector space, with a fixed nondegenerate symplectic or orthogonal form $\omega$. Let $G$ be the subgroup of $\operatorname{SL}(W)$ preserving $\omega$, then $G$ is either $\operatorname{SO}_{2n}, \operatorname{SO}_{2n+1}$, or $\operatorname{Sp}_{2n}$, depending on $\dim(W)$ and $\omega$. Let $\lambda$ be a dominant weight of $\mathfrak{sl}_n$, then Weyl's construction \cite{weyl1946classical} gives irreducible $G$-modules $\mathds{S}_{[\lambda]}(W)$. Note that for type $D$, irreducibility holds since we assume $\lambda_n = 0$. In the stable range $\ell(\mu)+ \ell(\nu)\leq n$, Koike--Terada \cite{koike1987young} proved tensor product multiplicities actually do not depend on $G$. In particular,
    \[ \mathds{S}_{[\mu]}(W) \otimes \mathds{S}_{[\nu]}(W) \cong \bigoplus_{\lambda} \mathds{S}_{[\lambda]}(W)^{\oplus N_{\mu,\nu,\lambda}}. \]
    Analogous to Schur functions, we have the Koike–Terada basis $\{s_{[\lambda]}\}$ of universal characters of $\mathds{S}_{[\lambda]}(W)$ for $\operatorname{Sp}$. These functions satisfy
    \[ s_{[\mu]}s_{[\nu]} = \sum_\lambda N_{\mu,\nu,\lambda}s_{[\lambda]}. \]

    Recently, the combinatorics of Newell--Littlewood numbers have been extensively studied in \cite{GOY21,gao2022newell,gao2025newell,min2025proof}. In terms of skew Schur functions, \cite[Proposition 2.3]{GOY21} showed that
    \[ N_{\mu,\nu,\lambda} = \sum_\alpha\langle s_{\mu/\alpha}s_{\nu/\alpha},s_{\lambda}\rangle, \]
    where the sum is over $\alpha\subseteq \mu\wedge\nu$. Thus, from Theorem \ref{thm:skew-schur-log}, the following analogue of Conjecture \ref{conj:log-concavity} for Koike--Terada positivity is immediate.

    \begin{thm}\label{thm:new-lit}
        Let $\lambda, \mu, \nu, \rho$ be dominant weights of $\mathfrak{sl}_n$, viewed as partitions ending in $0$.  If $\lambda + \mu = \nu + \rho$ and $\nu,\rho\in P_{\lambda,\mu}$, then $s_{[\nu]}s_{[\rho]} - s_{[\lambda]}s_{[\mu]}$ is Koike--Terada positive.
    \end{thm}
    
    Theorem~\ref{thm:new-lit} generalizes \cite[Theorem 7.4]{GOY21}.

\subsection{Shadow skew Schur functions}

    Recently, Chan--Chen--Pak--Soskin \cite{chan2026correlation} extended Lam--Postnikov--Pylyavskyy inequality in \cite{LPP07} to \textit{(skew) Ahlswede--Daykin--Schur inequality}. Among their many applications, they introduced \textit{shadow (skew) Schur functions} and generalized Lam--Postnikov--Pylyavskyy's Schur log-concavity results \cite[Section 3]{chan2026correlation}.
    
    Let $\lambda$ be a partition with at most $\ell$ nonzero parts, and $k\geq \ell$. Let
    \[
    \YY(\lambda;k,\ell) = \{\mu = (\mu_1,\ldots,\mu_k)~:~\mu_1 = \lambda_1,\ldots,\mu_\ell = \lambda_\ell\},
    \]
    and define the shadow Schur function to be
    \[
    S_{\lambda}^{(k,\ell)} = \sum_{\mu \in \YY(\lambda;k,\ell)} s_\mu.
    \]
    In particular, when $k = \ell$ or $\lambda_\ell = 0$, we have $\YY(\lambda;k,\ell) = \{\lambda\}$, so $S_{\lambda}^{(k,\ell)} = s_{\lambda}$. Shadow skew Schur functions are defined similarly:
    \[
    S_{\lambda/\mu}^{(k,\ell)} = \sum_{\lambda'\in \YY(\lambda;k,\ell), \mu' \in \YY(\mu;k,\ell)} s_{\lambda'/\mu'}.
    \]
    We further generalize their results \cite[Theorem 3.2 and 3.7]{chan2026correlation} and Theorem \ref{thm:skew-schur-log} to shadow skew Schur functions.

    \begin{thm}\label{thm:shadow-skew-schur-log}
        Let $\lambda,\mu,\nu,\rho, \ti\lambda,\ti\mu,\ti\nu,\ti\rho$ be partitions with at most $\ell$ nonzero parts, and $k\geq \ell$. If $(\lambda,\mu) + (\nu,\rho) = (\ti\lambda,\ti\mu) + (\ti\nu,\ti\rho)$, and $(\ti\lambda,\ti\mu),(\ti\nu,\ti\rho)\in P_{(\lambda,\mu),(\nu,\rho)}$, then $S^{(k,\ell)}_{\ti\lambda/\ti\mu} S^{(k,\ell)}_{\ti\nu/\ti\rho} - S^{(k,\ell)}_{\lambda/\mu} S^{(k,\ell)}_{\nu/\rho}$ is Schur positive.
    \end{thm}

    \begin{remark}
        Chan--Chen--Pak--Soskin also noted that their results also hold for \textit{lower Schur functions}, see \cite[Section 3.4]{chan2026correlation}. It would be interesting to see if our result holds for lower (skew) Schur functions as well.
    \end{remark}

\subsection{Peelable tableaux}

    Peelable tableaux were first used by Remmel--Whitney in \cite{remmel1984multiplying} to give the Littlewood--Richardson coefficients. Then, Nguyen--Nguyen--Woodruff in \cite{nguyen2025shuffle} introduced peelable tableaux for Nguyen--Pylyavskyy's shuffle tableaux \cite{nguyen2025temperley} as another attempt to prove Conjecture \ref{conj:log-concavity}. Peelable tableaux bear resemblance to skeps in the sense that the content of a peelable tableau has an interlacing property similar to the boundary of skeps. Thus, Speyer suspects that there are relations between skeps and shuffle tableaux. We explain this connection in Section \ref{sec:peelable} in a more general settings.

    A secret consequence of \cite[Theorem 1.4]{nguyen2025shuffle} is that besides Remmel--Whitney's and Nguyen--Nguyen--Woodruff's rules, we have more intermediate peelable tableaux rules for the Littlewood--Richardson coefficients. In addition, these rules are all related by Bender--Knuth involutions. On the other hand, besides our new skew hives and skew skeps rules, we have more intermediate skew hives rules. These are related by the octahedron recurrence.

    In Section \ref{sec:peelable}, via Gelfand--Tsetlin patterns, we will explain explicit bijections (Bijection \ref{bij:skew-hive-peelable}) between the peelable tableaux rules and skew hives rules.
    
\section*{Acknowledgement}

    We thank Alex Postnikov and David Speyer for helpful conversations.

\section{Hives, skeps, skew hives, skew skeps}

\subsection{Hives and skeps}

     Hives were invented by Knutson and Tao in \cite{knutson1999honeycomb} to compute Littlewood--Richardson coefficients. They have also been reintepreted as special cases of $M$-convex functions \cite{M03} or polymatroids over the tropical hyperfield $\mathbb{T}_0$ \cite{BHKKL25}. Let $\Delta_{n} = \{(i,j)\in\ZZ^2_{\geq 0}:i+j \leq n\}$.

    \begin{definition}\label{def:hives}
        A \textit{hive} is a function $h:\Delta_{n}\rightarrow\ZZ$ obeying the following inequalities for all $(i,j)$ such that the indices remain within $\Delta_{n}$.
        \begin{align*}
            h_{i,j} + h_{i+1,j} & \geq h_{i,j+1} + h_{i+1,j-1}, \\
            h_{i,j} + h_{i,j+1} & \geq h_{i+1,j} + h_{i-1,j+1}, \\
            h_{i+1,j} + h_{i,j+1} & \geq h_{i,j} + h_{i+1,j+1}.
        \end{align*}
        Given a hive $h$, define
        \begin{align*}
            h^{\uparrow} &= (h_{01}-h_{00}, h_{02}-h_{01},\ldots,h_{0,n} - h_{0,n-1}), \\
            h^{\leftarrow} &= (h_{n-1,0}-h_{n,0}, h_{n-2,0}-h_{n-1,0},\ldots,h_{0,0} - h_{1,0}), \\
            h^{\nwarrow} &= (h_{n-1,1}-h_{n,0}, h_{n-1,2}-h_{n-1,1},\ldots,h_{0,n} - h_{1,n-1}).
        \end{align*}
        Given partitions $\lambda,\mu,\nu$ of length $n$, let $\Hive_{\lambda,\mu}^\nu$ denotes the set of hives with $h^\leftarrow = \lambda$, $h^\uparrow = \mu$ and $h^\nwarrow = \nu$, and $h_{n,0} = 0$.
    \end{definition}

    \begin{thm}[\cite{knutson1999honeycomb, knutson2003positive}]
        Given partitions $\lambda,\mu,\nu$ of length $n$, then
        \[ c_{\lambda,\mu}^\nu = |\Hive_{\lambda,\mu}^\nu|. \]
    \end{thm}

    \begin{example}\label{ex:hives}
        Let $\lambda = (2,1,1)$, $\mu = (4,2,1,1)$, and $\nu = (5,3,2,2)$. There are two hives with the given boundaries
        \[ \begin{matrix}
            12 & & & & \\
            11 & 10 & & & \\
            10 & 9 & 8 & & \\
            8 & 8 & 7 & 5 & \\
            4 & 4 & 3 & 2 & 0
        \end{matrix}, \quad\quad\quad
        \begin{matrix}
            12 & & & & \\
            11 & 10 & & & \\
            10 & 9 & 8 & & \\
            8 & 7 & 6 & 5 & \\
            4 & 4 & 3 & 2 & 0
        \end{matrix}. \]
        Thus, $c_{(2,1,1),(4,2,1,1)}^{(5,3,2,2)} = 2$.
    \end{example}

    Let us recall the classical Littlewood--Richardson rule. Given a SSYT $T$, the reading word is the word obtained by reading the entries of $T$ from bottom to top and from left to right in each row. A word is a reverse lattice word if in any suffix, there are no more $i+1$ than $i$ for any $i$. A Yamanouchi tableau is a SSYT whose reading word is a reverse lattice word. Let $\LR_{\lambda,\mu}^\nu$ be the set of Yamanouchi tableaux of shape $\nu/\lambda$ and content $\mu$. Then, there is an easy bijection between $\LR_{\lambda,\mu}^\nu$ and $\Hive_{\lambda,\mu}^\nu$ \cite{buch2000saturation} as follows.

    \begin{bijection}[$\LR_{\lambda,\mu}^\nu \rightarrow \Hive_{\lambda,\mu}^\nu$]\label{bij:yama-hive}
        Given $T \in \LR_{\lambda,\mu}^\nu$,
        \begin{enumerate}
            \item fill the squares in $\lambda$ with $0$,
            \item let $\{x_{i,j}~|~0\leq i\leq n, 1\leq j\leq n\}$ be the Gelfand-Tsetlin pattern, where $(x_{i,1},x_{i,2},\ldots,x_{i,n})$ is the shape of of the subtableaux containing the entries $0,\ldots,i$,
            \item let $\{y_{i,j}~|~0\leq i, j\leq n\}$ be the row sum of $\{x_{i,j}\}$, where $y_{i,0} = 0$ and $y_{i,j} = \sum_{k=1}^j x_{i,k}$,
            \item the subpattern $\{y_{i,j}~|~i\leq j\}$ forms a hive in $\Hive_{\lambda,\mu}^\nu$.
        \end{enumerate}
    \end{bijection}

    \begin{example}
        Continuing Example \ref{ex:hives}, the two Yamanouchi tableaux of shape $(5,3,2,2)/(2,1,1)$ and content $(4,2,1,1)$ are
        \[ 
        \begin{ytableau} 
            \none & \none & 1 & 1 & 1 \\
            \none & 1 & 2 \\
            \none & 3 \\
            2 & 4
            \end{ytableau}, \quad\quad\quad
        \begin{ytableau} 
            \none & \none & 1 & 1 & 1 \\
            \none & 2 & 2 \\
            \none & 3 \\
            1 & 4
            \end{ytableau}.
        \]
        Padding $0$'s, we have
        \[ 
        \begin{ytableau} 
            0 & 0 & 1 & 1 & 1 \\
            0 & 1 & 2 \\
            0 & 3 \\
            2 & 4
            \end{ytableau}, \quad\quad\quad
        \begin{ytableau} 
            0 & 0 & 1 & 1 & 1 \\
            0 & 2 & 2 \\
            0 & 3 \\
            1 & 4
            \end{ytableau}.
        \]
        The Gelfand-Tsetlin patterns are
        \[ \begin{matrix}
            5 & & 3 & & 2 & & 2 \\
            & 5 & & 3 & & 2 & & 1 \\
            & & 5 & & 3 & & 1 & & 1 \\
            & & & 5 & & 2 & & 1 & & 0 \\
            & & & & 2 & & 1 & & 1 & & 0
        \end{matrix}, \quad\quad\quad
        \begin{matrix}
            5 & & 3 & & 2 & & 2 \\
            & 5 & & 3 & & 2 & & 1 \\
            & & 5 & & 3 & & 1 & & 1 \\
            & & & 5 & & 1 & & 1 & & 1 \\
            & & & & 2 & & 1 & & 1 & & 0
        \end{matrix}. \]
        The row sums are
        \[ \begin{matrix}
            \textcolor{red}{0} & & \textcolor{red}{5} & & \textcolor{red}{8} & & \textcolor{red}{10} & & 12 \\
            & \textcolor{red}{0} & & \textcolor{red}{5} & & \textcolor{red}{8} & & 10 & & 11 \\
            & & \textcolor{red}{0} & & \textcolor{red}{5} & & 8 & & 9 & & 10 \\
            & & & \textcolor{red}{0} & & 5 & & 7 & & 8 & & 8 \\
            & & & & 0 & & 2 & & 3 & & 4 & & 4
        \end{matrix}, \quad
        \begin{matrix}
            \textcolor{red}{0} & & \textcolor{red}{5} & & \textcolor{red}{8} & & \textcolor{red}{10} & & 12 \\
            & \textcolor{red}{0} & & \textcolor{red}{5} & & \textcolor{red}{8} & & 10 & & 11 \\
            & & \textcolor{red}{0} & & \textcolor{red}{5} & & 8 & & 9 & & 10 \\
            & & & \textcolor{red}{0} & & 5 & & 6 & & 7 & & 8 \\
            & & & & 0 & & 2 & & 3 & & 4 & & 4
        \end{matrix}. \]
        Removing the red entries and reorienting the grid, we get the two hives in Example \ref{ex:hives}.
    \end{example}

    Skeps were recently introduced by Speyer in \cite{speyer2026log} in the same spirit of hives.

    \begin{definition}\label{def:skeps}
        A \textit{skep} is a function $h:\Delta_{n}\rightarrow\ZZ$ obeying the following inequalities for all $(i,j)$ such that the indices remain within $\Delta_{n}$.
        \begin{align*}
            h_{i,j} + h_{i+1,j} & \geq h_{i,j+1} + h_{i+1,j-1}, \\
            h_{i,j} + h_{i,j+1} & \geq h_{i+1,j} + h_{i-1,j+1}, \\
            h_{i,j} + h_{i+1,j} & \geq h_{i,j-1} + h_{i+1,j+1}, \\
            h_{i,j} + h_{i,j+1} & \geq h_{i-1,j} + h_{i+1,j+1}, \\
            h_{11} & \geq h_{00}.
        \end{align*}
        Let $\Skep_{\lambda,\mu}^\nu$ denotes the set of skeps with $(h^\leftarrow, h^\uparrow) = (\lambda_1,\mu_1,\ldots,\lambda_n,\mu_n)$ and $h^\nwarrow = \nu$.
    \end{definition}

    \begin{thm}[{\cite[Theorem 3.16]{speyer2026log}}]
        Given partitions $\lambda,\mu,\nu$ of length $n$, then
        \[ c_{\lambda,\mu}^\nu = |\Skep_{\lambda,\mu}^\nu|. \]
    \end{thm}

    \begin{example}\label{ex:skeps}
        Let $\lambda = (2,1,1)$, $\mu = (4,2,1,1)$, and $\nu = (5,3,2,2)$ as in Example \ref{ex:hives}. There are two skeps with the given boundaries
        \[ \begin{matrix}
            12 & & & & \\
            11 & 10 & & & \\
            11 & 9 & 8 & & \\
            10 & 9 & 6 & 5 & \\
            9 & 7 & 6 & 2 & 0
        \end{matrix}, \quad\quad\quad
        \begin{matrix}
            12 & & & & \\
            11 & 10 & & & \\
            11 & 9 & 8 & & \\
            10 & 9 & 7 & 5 & \\
            9 & 7 & 6 & 2 & 0
        \end{matrix}, \]
        confirming $c_{(2,1,1),(4,2,1,1)}^{(5,3,2,2)} = 2$.
    \end{example}

    The sets $\Hive_{\lambda,\mu}^\nu,\Skep_{\lambda,\mu}^\nu,\Skep_{\mu,\lambda}^\nu,\Hive_{\mu,\lambda}^\nu$ are all in bijection with each other through the \textit{octahedron recurrence}. We follow the presentation in \cite{speyer2026log} here; see also \cite{henriques2006octahedron}.

    \begin{definition}
        Let $T_n = \{(i,j,t)~|~ 0 \leq i,j,|t|\leq n - i - j, t\equiv n+i+j~\text{mod}~2 \}$. A function $\Tilde{h}:T_n\rightarrow\ZZ$ is said to obey the \textit{octahedron recurrence} if, for all $(i,j,t+1)\in T_n$ such that $(i,j,t-1)\in T_n$, we have
        \begin{equation*}
            \Tilde{h}(i,j,t+1) + \Tilde{h}(i,j,t-1) = \begin{cases}
                \max\left(
                    \Tilde{h}(i-1,j,t) + \Tilde{h}(i+1,j,t),\atop\Tilde{h}(i,j-1,t) + \Tilde{h}(i,j+1,t)
                \right) &i,j\geq 1 \\
                \Tilde{h}(i-1,0,t) + \Tilde{h}(i+1,0,t) &i\geq 1, j = 0 \\
                \Tilde{h}(0,j-1,t) + \Tilde{h}(0,j+1,t) &i = 0, j\geq 1 \\
                \Tilde{h}(1,0,t) + \Tilde{h}(0,1,t) &i = j = 0
            \end{cases}
        \end{equation*}
        
        Let $\epsilon(i,j)$ be $0$ if $i+j \equiv n~\text{mod}~2$ and $1$ otherwise. We define the following subsets of $T_n$
        \begin{align*}
            S^{\ttop}_{\hive} &= \{(i,j,n-i-j):(i,j)\in\Delta_n\},\\
            S^{\bbot}_{\hive} &= \{(i,j,-n+i+j):(i,j)\in\Delta_n\},\\
            S^{\ttop}_{\skep} &= \{(i,j,\epsilon(i,j)):(i,j)\in\Delta_n\},\\
            S^{\bbot}_{\skep} &= \{(i,j,-\epsilon(i,j)):(i,j)\in\Delta_n\}.
        \end{align*}
    \end{definition}

    \begin{thm}[{\cite[Theorem 3.15]{speyer2026log}}]\label{thm:speyer-hive-skep}
        Let $\Tilde{h}:T_n\rightarrow\ZZ$ obeying the octahedron recurrence, then the following are equivalent
        \begin{enumerate}
            \item $\Tilde{h}|_{S^{\ttop}_{\hive}} \in \Hive_{\mu,\lambda}^\nu$,\\
            \item $\Tilde{h}|_{S^{\ttop}_{\skep}} \in \Skep_{\mu,\lambda}^\nu$,\\
            \item $\Tilde{h}|_{S^{\bbot}_{\skep}} \in \Skep_{\lambda,\mu}^\nu$,\\
            \item $\Tilde{h}|_{S^{\bbot}_{\hive}} \in \Hive_{\lambda,\mu}^\nu$.
        \end{enumerate}
    \end{thm}

    \begin{example}
        Consider the following function on $T_4$ obeying the octahedron recurrence.
        
        \begin{center}
        \begin{tabular}{*{10}{c}}
             & \raisebox{\height}{\scalebox{0.75}{$\left[\begin{matrix}
            \textcolor{red}{4}
        \end{matrix}\right]$}} & \raisebox{\height}{\scalebox{0.75}{$\left[\begin{matrix}
            \textcolor{red}{8} \\
            & \textcolor{red}{4}
        \end{matrix}\right]$}} & \raisebox{\height}{\scalebox{0.75}{$\left[\begin{matrix}
            \textcolor{red}{10} \\
             & \textcolor{red}{8} \\
            8 &  & \textcolor{red}{3}
        \end{matrix}\right]$}} & \raisebox{\height}{\scalebox{0.75}{$\left[\begin{matrix}
            \textcolor{orange}{11} \\
             & \textcolor{orange}{9} & \\
            \textcolor{blue}{10} &  & \textcolor{orange}{7} \\
             & \textcolor{blue}{7} &  & \textcolor{orange}{2}
        \end{matrix}\right]$}} & \raisebox{\height}{\scalebox{0.75}{$\left[\begin{matrix}
            \textcolor{orange}{12} & & & & \\
             & \textcolor{orange}{10} & & & \\
            \textcolor{blue}{11} &  & \textcolor{orange}{8} & & \\
             & \textcolor{blue}{9} &  & \textcolor{orange}{5} & \\
            \textcolor{blue}{9} &  & \textcolor{blue}{6} &  & \textcolor{orange}{0}
        \end{matrix}\right]$}} & \raisebox{\height}{\scalebox{0.75}{$\left[\begin{matrix}
            12 \\
             & 10 & \\
            10 &  & 7 \\
             & 8 &  & 4
        \end{matrix}\right]$}} & \raisebox{\height}{\scalebox{0.75}{$\left[\begin{matrix}
            11 \\
             & 9 \\
            9 &  & 6
        \end{matrix}\right]$}} & \raisebox{\height}{\scalebox{0.75}{$\left[\begin{matrix}
            10 \\
            & 7
        \end{matrix}\right]$}} & \raisebox{\height}{\scalebox{0.75}{$\left[\begin{matrix}
            8
        \end{matrix}\right]$}} \\
            $t$ & -4 & -3 & -2 & -1 & 0 & 1 & 2 & 3 & 4
        \end{tabular}
        \end{center}
        
        The orange and red entries form $S^{\bbot}_{\hive}$, which is the first hive in Example \ref{ex:hives}. The orange and blue entries form $S^{\bbot}_{\skep}$, which is the first skep in Example \ref{ex:skeps}.
    \end{example}

\subsection{Skew hives}

    \begin{definition}\label{def:skew-hive}
        A \textit{skew hive} is a function $h:\Delta_{2n}\rightarrow\ZZ$ obeying the following inequalities for all $(i,j)$ such that the indices remain within $\Delta_{2n}$.
        \begin{align*}
            h_{i,j} + h_{i+1,j} & \geq h_{i,j+1} + h_{i+1,j-1}, \\
            h_{i,j} + h_{i,j+1} & \geq h_{i+1,j} + h_{i-1,j+1}, \\
            h_{i+1,j} + h_{i,j+1} & \geq h_{i,j} + h_{i+1,j+1} \quad\quad\quad\text{for $i + j \neq n-1$}.
        \end{align*}
        Given a skew hive $h$, define
        \begin{align*}
            h^{\uparrow} &= (h_{01}-h_{00}, h_{02}-h_{01},\ldots,h_{0,2n} - h_{0,2n-1}), \\
            h^{\rightarrow} &= (h_{10}-h_{00}, h_{20}-h_{10},\ldots,h_{2n,0} - h_{2n-1,0}), \\
            h^{\nwarrow} &= (h_{2n-1,1}-h_{2n,0}, h_{2n-1,2}-h_{2n-1,1},\ldots,h_{0,2n} - h_{1,2n-1}).
        \end{align*}
        Let $\SkewHive_{\lambda/\mu,\nu/\rho}^\kappa$ be the set of skew hives with $h_{00} = 0$, $h^\uparrow = (\lambda,\nu)$, $h^\rightarrow = (\mu, \rho)$, and $h^\nwarrow = \kappa$.
    \end{definition}

    \begin{thm}\label{thm:skew_hive}
        Given partitions $\lambda,\mu,\nu,\rho$ of length $n$, and $\kappa$ of length $2n$, then
        \[ c_{\lambda/\mu, \nu/\rho}^\kappa = |\SkewHive_{\lambda/\mu,\nu/\rho}^\kappa|. \]
    \end{thm}

    \begin{example}\label{ex:skew-hives}
        For $\lambda = (2,2), \mu = (1,0), \nu = (4,3), \rho = (2,0)$, and $\kappa = (4,3,1)$, we have $c_{\lambda/\mu, \nu/\rho}^\kappa = 3$. The three skew hives with the given boundaries are
        \[ \begin{matrix}
            11 & & & & \\
            8 & 11 & & & \\
            4 & 7 & 10 & & \\
            2 & 3 & 6 & 7 & \\
            0 & 1 & 1 & 3 & 3
        \end{matrix}, \quad\quad\quad
        \begin{matrix}
            11 & & & & \\
            8 & 11 & & & \\
            4 & 8 & 10 & & \\
            2 & 3 & 6 & 7 & \\
            0 & 1 & 1 & 3 & 3
        \end{matrix}, \quad\quad\quad
        \begin{matrix}
            11 & & & & \\
            8 & 11 & & & \\
            4 & 8 & 10 & & \\
            2 & 3 & 7 & 7 & \\
            0 & 1 & 1 & 3 & 3
        \end{matrix}. \]
    \end{example}

    Denote $\lambda\ast_k \nu = (\lambda_1+k,\ldots,\lambda_n + k, \nu_1,\ldots,\nu_n)$, and consider the shape $(\lambda\ast_{\nu_1}\nu)/(\mu\ast_{\nu_1}\rho)$. Note that $s_{(\lambda\ast_{\nu_1}\nu)/(\mu\ast_{\nu_1}\rho)} = s_{\lambda/\mu}s_{\nu/\rho}$, so $c_{\lambda/\mu, \nu/\rho}^\kappa = |\Hive_{\mu\ast_{\nu_1}\rho,\kappa}^{\lambda\ast_{\nu_1}\nu}|$. Thus, to prove Theorem \ref{thm:skew_hive}, we need a bijection between $\Hive_{\mu\ast_{\nu_1}\rho,\kappa}^{\lambda\ast_{\nu_1}\nu}$ and $\SkewHive_{\lambda/\mu,\nu/\rho}^\kappa$.

    \begin{bijection}\label{bij:hive-skewhive}[$\Hive_{\mu\ast_{\nu_1}\rho,\kappa}^{\lambda\ast_{\nu_1}\nu} \rightarrow \SkewHive_{\lambda/\mu,\nu/\rho}^\kappa$]
        Given $h\in \Hive_{\mu\ast_{\nu_1}\rho,\kappa}^{\lambda\ast_{\nu_1}\nu}$, then define $\Tilde{h}:\Delta_{2n}\rightarrow\ZZ$ by
        \[ \Tilde{h}(i,j) = \begin{cases}
            h(2n-i-j,j) - (i+j)\nu_1 &\text{ if $0\leq i+j\leq n-1$}\\
            h(2n-i-j,j) - n\nu_1 &\text{ if $n\leq i+j\leq 2n$}.
        \end{cases} \]
        In other words, $\Tilde{h}$ is obtained from $h$ by first reversing the rows and then subtracting multiples of $\nu_1$ from each diagonal. Then $\Tilde{h}\in \SkewHive_{\lambda/\mu,\nu/\rho}^\kappa$.
    \end{bijection}

    \begin{example}
        For $\lambda,\mu.\nu,\rho$, and $\kappa$ as in Example \ref{ex:skew-hives}, the three hives in $\Hive_{\mu\ast_{\nu_1}\rho,\kappa}^{\lambda\ast_{\nu_1}\nu}$ are
        \[ 
        \begin{matrix}
            19 & & & & \\
            19 & 16 & & & \\
            18 & 15 & 12 & & \\
            15 & 14 & 11 & 6 & \\
            11 & 11 & 9 & 5 & 0
        \end{matrix}, \quad\quad\quad
        \begin{matrix}
            19 & & & & \\
            19 & 16 & & & \\
            18 & 16 & 12 & & \\
            15 & 14 & 11 & 6 & \\
            11 & 11 & 9 & 5 & 0
        \end{matrix}, \quad\quad\quad\begin{matrix}
            19 & & & & \\
            19 & 16 & & & \\
            18 & 16 & 12 & & \\
            15 & 15 & 11 & 6 & \\
            11 & 11 & 9 & 5 & 0
        \end{matrix}. \]
        Reversing each row, we have
        \[ 
        \begin{matrix}
            19 & & & & \\
            16 & 19 & & & \\
            12 & 15 & 18 & & \\
            6 & 11 & 14 & 15 & \\
            0 & 5 & 9 & 11 & 11
        \end{matrix}, \quad\quad\quad
        \begin{matrix}
            19 & & & & \\
            16 & 19 & & & \\
            12 & 16 & 18 & & \\
            6 & 11 & 14 & 15 & \\
            0 & 5 & 9 & 11 & 11
        \end{matrix}, \quad\quad\quad\begin{matrix}
            19 & & & & \\
            16 & 19 & & & \\
            12 & 16 & 18 & & \\
            6 & 11 & 15 & 15 & \\
            0 & 5 & 9 & 11 & 11
        \end{matrix}. \]
        Since $\nu_1 = 4$ and $n = 2$, we subtract $4$ from the diagonal $i+j = 1$, and $8$ from the diagonals $i+j = 2,3,4$. Then, we have
        \[ \begin{matrix}
            11 & & & & \\
            8 & 11 & & & \\
            4 & 7 & 10 & & \\
            2 & 3 & 6 & 7 & \\
            0 & 1 & 1 & 3 & 3
        \end{matrix}, \quad\quad\quad
        \begin{matrix}
            11 & & & & \\
            8 & 11 & & & \\
            4 & 8 & 10 & & \\
            2 & 3 & 6 & 7 & \\
            0 & 1 & 1 & 3 & 3
        \end{matrix}, \quad\quad\quad
        \begin{matrix}
            11 & & & & \\
            8 & 11 & & & \\
            4 & 8 & 10 & & \\
            2 & 3 & 7 & 7 & \\
            0 & 1 & 1 & 3 & 3
        \end{matrix}. \]
        These are exactly the three skew hives in Example \ref{ex:skew-hives}.
    \end{example}

    \begin{proof}[Proof of Theorem \ref{thm:skew_hive}]
        It suffices to prove that Bijection \ref{bij:hive-skewhive} is a bijection. First, we verify that $\Tilde{h}$ has the correct boundary. Let us verify for $\Tilde{h}^\nwarrow$; the other boundaries are similar.
        \begin{align*}
            \Tilde{h}^\nwarrow &= (\Tilde{h}_{2n-1,1} - \Tilde{h}_{2n,0},\ldots,\Tilde{h}_{0,2n} - \Tilde{h}_{1,2n-1}) \\
            &= ((h_{01}-n\nu)-(h_{00}-n\nu),\ldots, (h_{0,n}-n\nu) - (h_{0,n-1}-n\nu) ) \\
            &= h^\uparrow = \kappa.
        \end{align*}
        Now we check the inequalities. We have
        \begin{align*}
            &~\Tilde{h}_{i,j} + \Tilde{h}_{i+1,j} \geq \Tilde{h}_{i,j+1} + \Tilde{h}_{i+1,j-1}\\ \Longleftrightarrow& ~(h_{2n-i-j,j} - k_1\nu_1) + (h_{2n-i-j-1,j} - k_2\nu_1) \geq (h_{2n-i-j-1,j+1} - k_1\nu_1) + (h_{2n-i-j,j-1} -k_2\nu_1)\\ \Longleftrightarrow& ~h_{2n-i-j,j} + h_{2n-i-j-1,j}\geq h_{2n-i-j-1,j+1} + h_{2n-i-j,j-1},
        \end{align*}
        where $k_1$ and $k_2$ depend on $2n-i$ and $2n-i-1$, respectively. Note that the points $(2n-i-j,j), (2n-i-j-1,j), (2n-i-j-1,j+1), (2n-i-j,j-1)$ are all in $\Delta_{2n}$ if and only if the points $(i,j), (i+1,j), (i,j+1), (i+1,j-1)$ are also all in $\Delta_{2n}$, so $\Tilde{h}$ satisfies $\Tilde{h}_{i,j} + \Tilde{h}_{i+1,j} \geq \Tilde{h}_{i,j+1} + \Tilde{h}_{i+1,j-1}$ for all $i,j$ if and only if the same is true for $h$.

        Similarly, we can show that $\Tilde{h}$ satisfies $\Tilde{h}_{i,j} + \Tilde{h}_{i,j+1} \geq \Tilde{h}_{i+1,j} + \Tilde{h}_{i-1,j+1}$ for all $i,j$ if and only if $h$ satisfies $h_{i+1,j} + h_{i,j+1} \geq h_{i,j} + h_{i+1,j+1}$ for all $i,j$,
        and $\Tilde{h}$ satisfies $\Tilde{h}_{i+1,j} + \Tilde{h}_{i,j+1} \geq \Tilde{h}_{i,j} + \Tilde{h}_{i+1,j+1}$ for all $i+j\neq n-1$ if and only if $h$ satisfies $h_{i,j} + h_{i,j+1} \geq h_{i+1,j} + h_{i-1,j+1}$ for all $i+j\neq n+1$.

        Finally, we need to prove that if $\Tilde{h}$ satisfies all required inequalities, then $h$ satisfies $h_{i,j} + h_{i,j+1} \geq h_{i+1,j} + h_{i-1,j+1}$ for all $i+j= n+1$, which is equivalent to $\Tilde{h}$ satisfies $\Tilde{h}_{i+1,j} + \Tilde{h}_{i,j+1} \geq \Tilde{h}_{i,j} + \Tilde{h}_{i+1,j+1} - \nu_1$ for all $i+j = n-1$. We have for $i+j = n-1$,
        \begin{align*}
            \Tilde{h}_{i+1,j} - \Tilde{h}_{i,j} &\geq \Tilde{h}_{i+2,j-1} - \Tilde{h}_{i+1,j-1}\\
            &\geq \cdots\\
            &\geq \Tilde{h}_{n,0} - \Tilde{h}_{n-1,0}\\
            &\geq 0.
        \end{align*}
        On the other hand,
        \begin{align*}
            \Tilde{h}_{i+1,j+1} - \Tilde{h}_{i,j+1} &\leq \Tilde{h}_{i,j+2} - \Tilde{h}_{i-1,j+2}\\
            &\leq\cdots\\
            &\leq \Tilde{h}_{1,n} - \Tilde{h}_{0,n}.
        \end{align*}
        Note that the subset $\{\Tilde{h}_{i,j}~|~j\geq n\}$ itself forms a hive, so we must have
        \[ \Tilde{h}_{1,n} - \Tilde{h}_{0,n} \leq \Tilde{h}_{0,n+1} - \Tilde{h}_{0,n} = \nu_1. \]
        Thus,
        \[ \Tilde{h}_{i+1,j} - \Tilde{h}_{i,j} \geq 0 \geq \Tilde{h}_{i+1,j+1} - \Tilde{h}_{i,j+1} - \nu_1, \]
        as expected.
    \end{proof}

\subsection{Skew skeps}

    \begin{definition}\label{def:skew-skep}
        A \textit{skew skep} is a function $h:\Delta_{2n}\rightarrow\ZZ$ obeying the following inequalities for all $(i,j)$ such that the indices remain within $\Delta_{2n}$.
        \begin{align*}
            h_{i,j} + h_{i+1,j} & \geq h_{i,j+1} + h_{i+1,j-1}, \\
            h_{i,j} + h_{i,j+1} & \geq h_{i+1,j} + h_{i-1,j+1}, \\
            h_{i,j} + h_{i+1,j} & \geq h_{i,j-1} + h_{i+1,j+1}, \\
            h_{i,j} + h_{i,j+1} & \geq h_{i-1,j} + h_{i+1,j+1}.
        \end{align*}
        Let $\SkewSkep_{\lambda/\mu,\nu/\rho}^\kappa$ be the set of skew skeps with $h_{00} = 0$, $h^\uparrow = (\lambda_1,\nu_1,\lambda_2,\nu_2,\ldots,\lambda_n,\nu_n)$, $h^\rightarrow = (\mu_1,\rho_1,\mu_2,\rho_2\,\ldots,\mu_n,\rho_n)$, and $h^\nwarrow = \kappa$.
    \end{definition}

    \begin{thm}\label{thm:skew_skep}
        Given partitions $\lambda,\mu,\nu,\rho$ of length $n$, and $\kappa$ of length $2n$, then
        \[ c_{\lambda/\mu, \nu/\rho}^\kappa = |\SkewSkep_{\lambda/\mu,\nu/\rho}^\kappa|. \]
    \end{thm}

    \begin{proof}[Proof teaser]
        Theorem \ref{thm:skew_skep} will follow from Theorem \ref{thm:skew-hive-skep}, which is an analogue of Theorem \ref{thm:speyer-hive-skep}. 
    \end{proof}    

    \begin{example}\label{ex:skew-skeps}
        For $\lambda = (2,2), \mu = (1,0), \nu = (4,3), \rho = (2,0)$, and $\kappa = (4,3,1)$, we have $c_{\lambda/\mu, \nu/\rho}^\kappa = 3$. The three skew skeps with the given boundaries are
        \[ \begin{matrix}
            11 & & & & \\
            8 & 11 & & & \\
            6 & 7 & 10 & & \\
            2 & 5 & 6 & 7 & \\
            0 & 1 & 3 & 3 & 3
        \end{matrix}, \quad\quad\quad
        \begin{matrix}
            11 & & & & \\
            8 & 11 & & & \\
            6 & 8 & 10 & & \\
            2 & 6 & 6 & 7 & \\
            0 & 1 & 3 & 3 & 3
        \end{matrix}, \quad\quad\quad
        \begin{matrix}
            11 & & & & \\
            8 & 11 & & & \\
            6 & 8 & 10 & & \\
            2 & 6 & 7 & 7 & \\
            0 & 1 & 3 & 3 & 3
        \end{matrix}. \]
    \end{example}

    The strategy for proving Theorem \ref{thm:skew_skep} is similar to Theorem \ref{thm:speyer-hive-skep} with slightly different regions.

    \begin{definition}
        Let $\Pi_n = \{(i,j,t)~|~|t|\leq n; 0 \leq i,j\leq 2n - i - j; |t|\leq i+j \leq 2n-|t|; t\equiv i+j~\text{mod}~2 \}$.
        Let $\Tilde{\epsilon}(i,j)$ be $0$ if $i+j \equiv 0~\text{mod}~2$ and $1$ otherwise. We define the following subsets of $\Pi_n$
        \begin{align*}
            S^{\ttop}_{\skewhive} &= \{(i,j,i+j):(i,j)\in\Delta_{2n},i+j \leq n\} \\
            &\quad\cup \{(i,j,2n-i-j):(i,j)\in\Delta_{2n},i+j \geq n\},\\
            S^{\bbot}_{\skewhive} &= \{(i,j,-i-j):(i,j)\in\Delta_{2n},i+j \leq n\} \\
            &\quad\cup \{(i,j,-2n+i+j):(i,j)\in\Delta_{2n},i+j \geq n\},\\
            S^{\ttop}_{\skewskep} &= \{(i,j,\epsilon(i,j)):(i,j)\in\Delta_{2n}\},\\
            S^{\bbot}_{\skewskep} &= \{(i,j,-\epsilon(i,j)):(i,j)\in\Delta_{2n}\}.
        \end{align*}
    \end{definition}

    \begin{example}
        Consider the following function on $\Pi_2$ obeying the octahedron recurrence.

        \begin{center}
        \begin{tabular}{*{10}{c}}
             & \raisebox{\height}{\scalebox{0.85}{$\left[\begin{matrix}
            \textcolor{red}{4} \\
             & \textcolor{red}{3} \\
             &  & \textcolor{red}{1}
        \end{matrix}\right]$}} & \raisebox{\height}{\scalebox{0.85}{$\left[\begin{matrix}
            \textcolor{orange}{8} \\
             & \textcolor{orange}{7} & \\
            \textcolor{orange}{2} &  & \textcolor{orange}{6} \\
             & \textcolor{orange}{1} &  & \textcolor{orange}{3}
        \end{matrix}\right]$}} & \raisebox{\height}{\scalebox{0.85}{$\left[\begin{matrix}
            \textcolor{orange}{11} & & & & \\
             & \textcolor{orange}{11} & & & \\
             \textcolor{blue}{6} &  & \textcolor{orange}{10} & & \\
             & \textcolor{blue}{5} &  & \textcolor{orange}{7} & \\
            \textcolor{orange}{0} &  & \textcolor{blue}{3} &  & \textcolor{orange}{3}
        \end{matrix}\right]$}} & \raisebox{\height}{\scalebox{0.85}{$\left[\begin{matrix}
            9 \\
             & 9 & \\
            4 &  & 7 \\
             & 2 &  & 3
        \end{matrix}\right]$}} & \raisebox{\height}{\scalebox{0.85}{$\left[\begin{matrix}
            7 \\
             & 6 \\
             &  & 2
        \end{matrix}\right]$}} \\
            $t$ & -2 & -1 & 0 & 1 & 2
        \end{tabular}
        \end{center}
        
        The orange and red entries form $S^{\bbot}_{\skewhive}$, which is the first hive in Example \ref{ex:skew-hives}. The orange and blue entries form $S^{\bbot}_{\skewskep}$, which is the first skep in Example \ref{ex:skew-skeps}.
    \end{example}

    Before stating Theorem \ref{thm:skew-hive-skep}, we need one more detail, which explains where the inequalities come from.

    \begin{definition}\label{def:unit-triangle-rhombus}
        A \textit{unit triangle} is an equilateral triangle whose vertices are in $\Pi_n$ and edges have length $\sqrt{2}$. In other words, the vertices are $\{w\pm(1,0,0),w\pm(0,1,0),w\pm(0,0,1)\}$ for some $w\in \ZZ^3$ and choices of signs. A \textit{unit rhombus} consists of two coplanar adjacent unit triangles, bordering along an edge. In a unit rhombus, we call the line segment between the $60^\circ$ angles the \textit{long diagonal} and the line segment between the $120^\circ$ angles the \textit{short diagonal}.

        Let $h:\Pi_n\rightarrow\ZZ$ be a function, and let $R$ be a unit rhombus with long diagonal $(x, y)$ and short diagonal $(x',y')$. We say $h$ satisfies the \textit{rhombus inequality at $R$} if $h(x')+h(y')\geq h(x) + h(y)$. We say $h$ satisfies all rhombus inequalities if $h$ satisfies the rhombus inequality at every unit rhombus in $\Pi_n$.
    \end{definition}

    Figure \ref{fig:t-coor} shows the $t$-coordinates of $S^{\bbot}_{\skewhive}$ and $S^{\bbot}_{\skewskep}$ in $\Pi_n$. The black bold edges are short edges of the green rhombi. The readers can check that the rhombi inequalities in Figure \ref{subfig:skew-hive-coor} (resp. Figure \ref{subfig:skew-skep-coor}) match the inequality in Definition \ref{def:skew-hive} (resp. Definition \ref{def:skew-skep}).

    \begin{figure}[h!]
     \centering
        \begin{subfigure}[b]{0.4\textwidth}
            \centering
            \includegraphics[scale = 0.9]{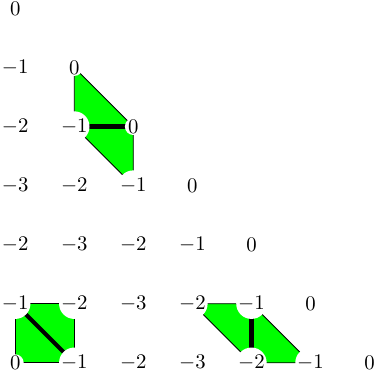}
            \caption{Skew hive}
            \label{subfig:skew-hive-coor}
        \end{subfigure}
     \quad\quad
        \begin{subfigure}[b]{0.4\textwidth}
            \centering
            \includegraphics[scale = 0.9]{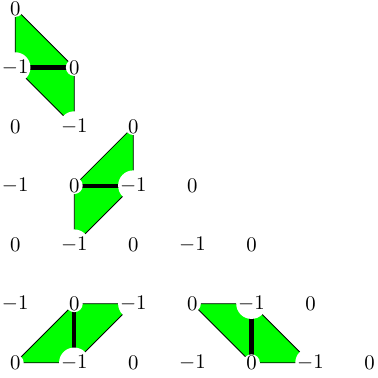}
            \caption{Skew skep}
            \label{subfig:skew-skep-coor}
        \end{subfigure}

        \caption{$t$-coordinates and rhombi of skew hives and skew skeps in $\Pi_n$}
        \label{fig:t-coor}
    \end{figure}

    \begin{thm}\label{thm:skew-hive-skep}
        Let $\Tilde{h}:\Pi_n\rightarrow\ZZ$ obeying the octahedron recurrence, then the following are equivalent
        \begin{enumerate}
            \item $\Tilde{h}|_{S^{\ttop}_{\skewhive}} \in \SkewHive_{\nu/\rho, \lambda/\mu}^\kappa$,\\
            \item $\Tilde{h}|_{S^{\ttop}_{\skewskep}} \in \SkewSkep_{\nu/\rho, \lambda/\mu}^\kappa$,\\
            \item $\Tilde{h}|_{S^{\bbot}_{\skewskep}} \in \SkewSkep_{\lambda/\mu, \nu/\rho}^\kappa$,\\
            \item $\Tilde{h}|_{S^{\bbot}_{\skewhive}} \in \SkewHive_{\lambda/\mu, \nu/\rho}^\kappa$,\\
            \item $\Tilde{h}$ satisfies all rhombus inequalities.
        \end{enumerate}
    \end{thm}

    To prove Theorem \ref{thm:skew-hive-skep}, we will need more definitions. These notions are adapted, with suitable changes, from \cite{henriques2006octahedron}.

\subsection{Sections and wavefronts}

    \begin{definition}
        Let
        \[ \RR\Delta_{2n} := \text{Hull}\{(0,0),(0,2n),(2n,0)\}, \]
        and
        \[ \RR \Pi_n := \text{Hull}\{(0,0,0),(0,n,n),(n,0,n),(0,n,-n),(n,0,-n),(0,2n,0),(2n,0,0)\}. \]
    \end{definition}

    \begin{definition}
        A \textit{section} $\SSS$ is a union of unit triangles in $\RR \Pi_n$ such that the projection $\operatorname{proj}:\SSS\rightarrow \RR\Delta_{2n}$ is bijective. We emphasize that a section is \emph{not} determined by the discrete set $\SSS\cap \Z^3$. For example, see the sections $\SSS_1$ and $\SSS_2$ in Figure~\ref{fig:S-triangles}.
    \end{definition}

    \begin{example}\label{ex:Sbot}
        Let
        \[
        \SSS_{\bbot} = 
        \]
        \[
        \text{Hull}\{(0,0,0),(0,n,-n),(n,0,-n)\} \bigcup \text{Hull}\{(0,n,-n),(n,0,-n),(0,2n,0),(2n,0,0)\},
        \]
        then $\SSS_{\bbot}$ is a section. In particular, $\SSS_{\bbot}$ is a union of the triangles in Figure \ref{subfig:Sbot-triangles}.

        \begin{figure}[h!]
           \centering
            \begin{subfigure}[b]{0.3\textwidth}
                \centering
                \includegraphics[scale = 0.7]{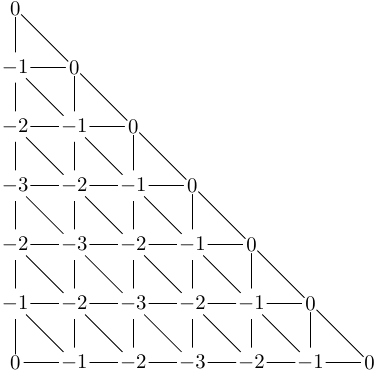}
                \caption{$\SSS_{\bbot}$}
                \label{subfig:Sbot-triangles}
            \end{subfigure}
         \quad
            \begin{subfigure}[b]{0.3\textwidth}
                \centering
                \includegraphics[scale = 0.7]{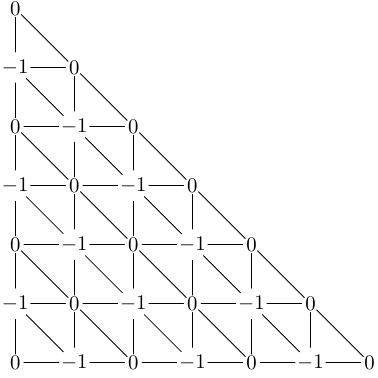}
                \caption{$\SSS_1$}
                \label{subfig:S1-triangles}
            \end{subfigure}
         \quad
            \begin{subfigure}[b]{0.3\textwidth}
                \centering
                \includegraphics[scale = 0.7]{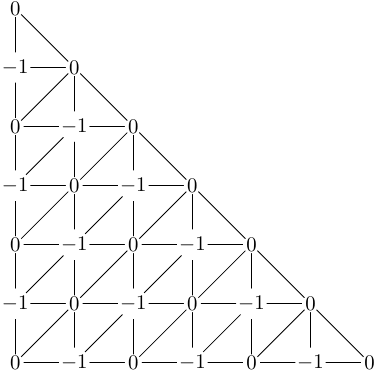}
                \caption{$\SSS_2$}
                \label{subfig:S2-triangles}
            \end{subfigure}

            \caption{Sections}
            \label{fig:S-triangles}
        \end{figure}
    \end{example}

    \begin{definition}
         A \textit{wavefront} $\WW$ is a two-dimensional subset of $\RR \Pi_n$ of one of these types:
         \begin{itemize}
             \item[(1)] The intersection of $\RR\Pi_n$ with $\{t-c = i-j\}$, for $c$ even, $|c| < 2n$.
             \item[(2)] The intersection of $\RR\Pi_n$ with $\{t-c = j-i\}$, for $c$ even, $|c| < 2n$.
             \item[(3)] The intersection of $\RR\Pi_n$ with $\{|t-c| = i+j\}$, for $c$ even, $0 < |c| < 2n$.
         \end{itemize}

         We say that a wavefront $\WW$ and a section $\SSS$ \textit{intersect transversely} if $\WW\cap \SSS$ has dimension $\leq 1$ and in case $\WW$ is of type (3) above, the intersection does not contain the \textit{cut point } $(0,0,c)$.
    \end{definition}

    \begin{example}
        One can check that the section $\SSS_{\bbot}$ in Example \ref{ex:Sbot} intersects transversely to all wavefronts. For example, for $n=3$, the intersection of $\SSS_{\bbot}$ with the wavefronts are shown in Figure \ref{fig:intersect-type}.

        \begin{figure}[h!]
         \centering
            \begin{subfigure}[b]{0.3\textwidth}
                \centering
                \includegraphics[scale = 0.7]{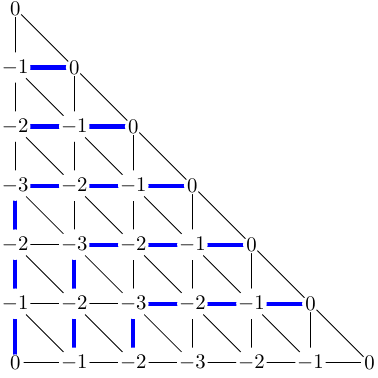}
                \caption{Type (1)}
                \label{intersect-type1}
            \end{subfigure}
         \quad
            \begin{subfigure}[b]{0.3\textwidth}
                \centering
                \includegraphics[scale = 0.7]{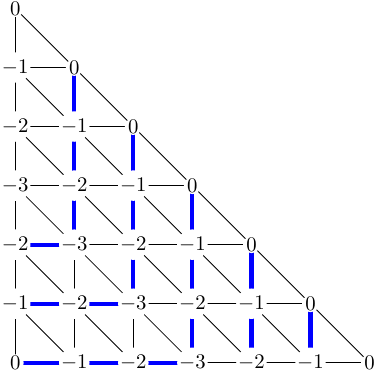}
                \caption{Type (2)}
                \label{intersect-type2}
            \end{subfigure}
         \quad
            \begin{subfigure}[b]{0.3\textwidth}
                \centering
                \includegraphics[scale = 0.7]{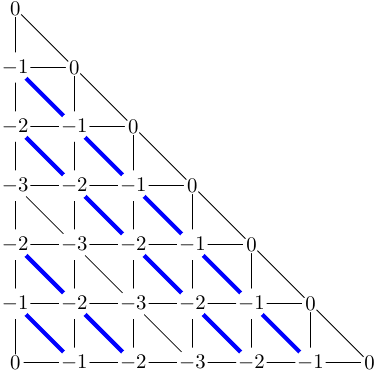}
                \caption{Type (3)}
                \label{intersect-type3}
            \end{subfigure}

            \caption{Intersection of $\SSS_{\bbot}$ with the wavefronts}
            \label{fig:intersect-type}
        \end{figure}
    \end{example}

    \begin{lemma}\label{lm:tranverse}
        For any unit rhombus $R$, there is a section $\SSS_0$ containing $R$ and a wavefront $\WW$ transverse to $\SSS_0$.
    \end{lemma}

    \begin{proof}
        This is implicit in \cite{henriques2006octahedron} and \cite{speyer2026log}. The only caveat is that for wavefronts of type (3), we do not use the wavefront $\{|t| = i+j\}$. This is because this wavefront does not intersect the interior of $\RR\Pi_n$ and hence is not relevant.
    \end{proof}

    \begin{lemma}\label{lem:any-rhombus-int}
        Let $\Tilde{h}:\Pi_n\rightarrow\ZZ$ obeying the octahedron recurrence. Suppose for every wavefront $\WW$, there is a section $\SSS(\WW)$ transverse to $\WW$ such that $\Tilde{h}$ satisfies all rhombus inequalities along $\WW\cap\SSS(\WW)$. Then $\Tilde{h}$ satisfies all rhombus inequalities.
    \end{lemma}

    \begin{proof}
        This is verbatim to \cite[Lemma 3.26]{speyer2026log}, where Speyer explained how his Lemma 3.26 follows from \cite[Lemma 3.1]{henriques2006octahedron}. We include the explanation here for completeness.

        Since the section $\SSS_{\bbot}$ intersects tranversely to every wavefront $\WW$, it follows from \cite[Lemma 3.1]{henriques2006octahedron} that if two sections $\SSS$ and $\SSS'$ are tranverse to $\WW$, and $\Tilde{h}$ satisfies rhombus inequalities along $\WW\cap\SSS$, then $\Tilde{h}$ also satisfies rhombus inequalities along $\WW\cap\SSS'$.

        For any rhombus $R$, Lemma \ref{lm:tranverse} says that there is a section $\SSS_0$ containing $R$ and a wavefront $\WW$ transverse to $\SSS_0$. Thus, if there is a section $\SSS(\WW)$ as above, then $\Tilde{h}$ satisfies all rhombus inequalities along $\WW\cap \SSS_0$. In particular, $\Tilde{h}$ satisfies the rhombus inequality at $R$.
    \end{proof}

    Now we are ready to prove Theorem \ref{thm:skew-hive-skep}.

    \begin{proof}[Proof of Theorem \ref{thm:skew-hive-skep}]
        The proof is similar to that of 
        \cite[Theorem 3.15]{speyer2026log}.
    
        \textbf{The easy direction:} It is easy to see that (5) implies (1)-(4) since all inequalities for skew hives and skew skeps are rhombus inequalities.

        \textbf{The other direction:} We first show that (4) implies (5). The argument for (1) implies (5) is similar. To show (4) implies (5), recall that for every wavefront $\WW$, the section $\SSS_{\bbot}$ intersect $\WW$ transversely. Furthermore, all rhombus inequalities on along $\WW\cap \SSS_{\bbot}$ are hive inequalities, so $\SSS_{\bbot}$ meets the condition of Lemma \ref{lem:any-rhombus-int}.

        Now we show (3) implies (5); the argument for (2) implies (5) is similar. We use the sections $\SSS_1$ and $\SSS_2$ similar to \cite[Definition 3.22]{speyer2026log}. These sections can be seen in Figures \ref{subfig:S1-triangles} and \ref{subfig:S2-triangles}. We leave the explicit descriptions to the readers. One can check that $\SSS_1$ intersects wavefronts of type (1) and (2) transversely (Figure \ref{fig:intersect-S1}), and $\SSS_2$ intersects wavefronts of type (3) transversely (Figure \ref{fig:intersect-S2}). Thus, $\SSS_1$ and $\SSS_2$ combined meet the condition of Lemma \ref{lem:any-rhombus-int}.

        \begin{figure}[h!]
         \centering
            \begin{subfigure}[b]{0.4\textwidth}
                \centering
                \includegraphics[scale = 0.9]{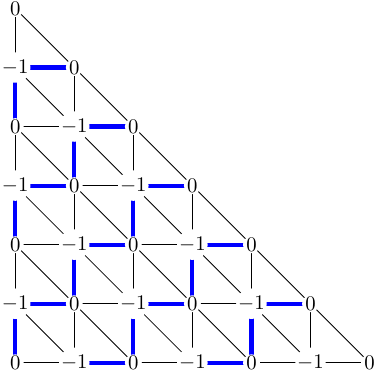}
                \caption{Type (1)}
                \label{intersect-S1-type1}
            \end{subfigure}
         \quad\quad
            \begin{subfigure}[b]{0.4\textwidth}
                \centering
                \includegraphics[scale = 0.9]{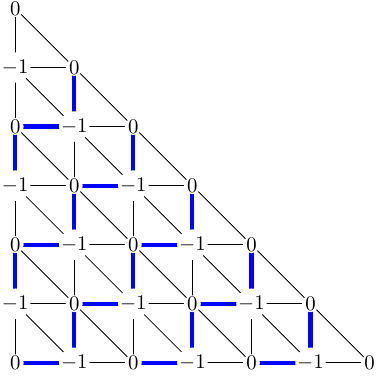}
                \caption{Type (2)}
                \label{intersect-S1-type2}
            \end{subfigure}

            \caption{Intersection of $\SSS_{1}$ with the wavefronts of types (1) and (2)}
            \label{fig:intersect-S1}
        \end{figure}

        \begin{figure}[h!]
         \centering
            \begin{subfigure}[b]{0.4\textwidth}
                \centering
                \includegraphics[scale = 0.9]{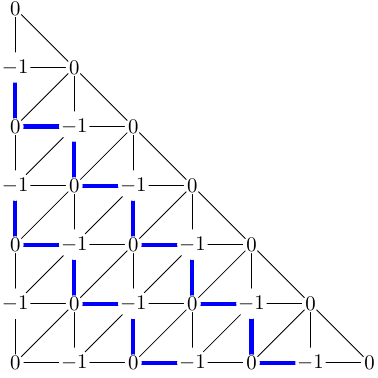}
                \caption{Type (3) positive $c$}
                \label{intersect-S2-type3pos}
            \end{subfigure}
         \quad\quad
            \begin{subfigure}[b]{0.4\textwidth}
                \centering
                \includegraphics[scale = 0.9]{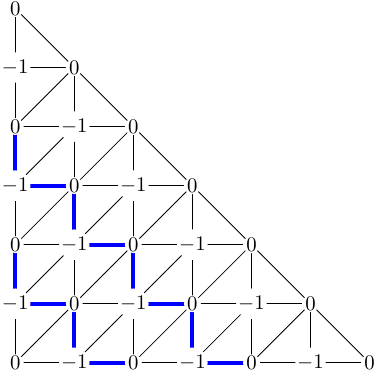}
                \caption{Type (3) negative $c$}
                \label{intersect-S2-type3neg}
            \end{subfigure}

            \caption{Intersection of $\SSS_{2}$ with the wavefronts of type (3)}
            \label{fig:intersect-S2}
        \end{figure}

        \textbf{The boundary:} Let $h_1, h_2,h_3,h_4$ be $\Tilde{h}|_{S^{\ttop}_{\skewhive}},\Tilde{h}|_{S^{\ttop}_{\skewskep}},\Tilde{h}|_{S^{\bbot}_{\skewskep}},\Tilde{h}|_{S^{\bbot}_{\skewhive}}$, respectively. Then,
        \[
        h_1^{\nwarrow} = h_2^{\nwarrow} = h_3^{\nwarrow} = h_4^{\nwarrow}.
        \]
        By similar analysis of the octahedron recurrence as \cite[Figure 3.1]{speyer2026log}, one can show that
        \[
        h_1^{\uparrow} = (\nu,\lambda) \Longleftrightarrow h_2^{\uparrow} = (\nu_1,\lambda_1,\ldots,\nu_n,\lambda_n) \Longleftrightarrow h_3^{\uparrow} = (\lambda_1,\nu_1,\ldots,\lambda_n,\nu_n) \Longleftrightarrow h_4^{\uparrow} = (\lambda,\nu),
        \]
        and
        \[
        h_1^{\rightarrow} = (\rho,\mu) \Longleftrightarrow h_2^{\rightarrow} = (\rho_1,\mu_1,\ldots,\rho_n,\mu_n) \Longleftrightarrow h_3^{\rightarrow} = (\mu_1,\rho_1,\ldots,\mu_n,\rho_n) \Longleftrightarrow h_4^{\rightarrow} = (\mu,\rho).
        \]
    \end{proof}

\section{Skew Schur log-concavity}
\subsection{$L$-convexity}
Murota \cite{M03} developed two dual notions of discrete convexity, which he termed $M$-convexity and $L$-convexity. Here we recall some facts about $L$-convexity.
\begin{definition}
    For $u, v\in \Z^n,$ we define \[L(u, v) := \{z\in \Z^n\mid \min(u_i - u_j, v_i - v_j)\le z_i - z_j\le \min(u_i - u_j, v_i - v_j)\text{ for all }1\le i < j\le n\}.\]
\end{definition}
This coincides with the notation $P_{u,v}$ as defined in the introduction.

\begin{definition}[{c.f. \cite[Definition 1.2 and 4.5]{speyer2026log}}]
    A set $S\sbs \Z^n$ is $L$-convex if \begin{itemize}
        \item $u\in S\implies u\pm 1^n\in S.$
        \item $u, v\in S\implies L(u, v)\sbs S.$
    \end{itemize}
\end{definition}

While this not the original definion introduced by Murota {\cite[Definition 5.1]{M03}}, Speyer have shown that it is in fact equivalent \cite[Theorem 4.3]{speyer2026log}. We note that intersections of $L$-convex sets are $L$-convex, and $L(u,v)$ is the $L$-convex hull of $u$ and $v$ in the sense that it is the minimal $L$-convex set containing $u$ and $v.$ 

\begin{definition}[{c.f. \cite[Definitions 1.3 and 4.5]{speyer2026log}}]
    A function $f\colon \Z^n\to \RR_{\ge 0}$ is $L$-log-concave if \begin{itemize}
        \item There exists $c\in \RR_{>0}$ such that $f(u+1^n) = cf(u)$ for all $u.$
        \item For $u, v, u', v'$ such that $u', v'\in L(u, v)$ and $u+v = u'+v'$, we have $f(u')f(v')\ge f(u)f(v).$
    \end{itemize}
\end{definition}

The definition above is also not Murota's original definition; but it is equivalent by 
\cite[Theorem 4.4]{speyer2026log}. Note that a set $S$ is $L$-convex if and only if $1_S$ (the indicator function of $S$) is $L$-log-concave.

Since $L$ convex set is invariant under translation by $1^n$, it is natural to consider the image of the $L$-convex set in $\Z^n/1^n$. One way to do this is to intersect the $L$-convex set with the hyperplane $\{(x_1, \dots, x_n)\mid x_n = 0\}$, which gives a set in $\Z^{n-1}.$ Such a set is said to be $L^\natural$-convex by Murota \cite{M03}. It is clear that $S\sbs \Z^{n-1}$ is $L^\natural$ convex if and only if $\{(s, 0)+t1^n\mid s\in S, t\in \Z\}\sbs \Z^n$ is $L$-convex. $L^\natural$ convex set coincides with what is known as an alcolved polyhedra in the sense of \cite{lam2007alcoved,lam2018alcoved}. Similarly, a function $f\colon \Z^{n-1}\to \RR_{\ge 0}$ is said to be $L^\natural$-log-concave if it is $L$-log-concave after extending it to $\Z^n$ by requiring $f(u+1^n) = cf(u)$ for some $c\in \RR_{>0}$ (in fact, this property does not depends on the choice of $c$.) 

One can intrinsically characterize $L^\natural$-convex set and $L^\natural$-log-concave function as follows.
\begin{definition}
    For $u, v\in \Z^m,$ we define \begin{align}L^{\natural}(u, v) &:= \{z\in \Z^m\mid (z, 0)\in L((u, 0), (v, 0))\text{ in }\Z^{m+1}\}\notag \\
    &= \{z\in \Z^m\mid z\in L(u, v), \min(u_i, v_i)\le z_i \le \max(u_i, v_i)\text{ for all }i\}.\end{align}
\end{definition}
\begin{definition}
    A set $S\sbs \Z^m$ is said to be  $L^\natural$-convex if for all $u, v\in S,$ we have $L^\natural(u,v)\sbs S.$
\end{definition}
\begin{definition}
    A function $f\colon \Z^m\to \RR_{\ge0}$ is said to be  $L^\natural$-log-concave if for all $u, v, u', v'$ such that $u', v'\in L^{\natural}(u, v)$ and $u' + v' = u+v$, we have $f(u')f(v')\ge f(u)f(v).$  
\end{definition}

While $L$ and $L^\natural$ are conceptually the same notion, it is worth noting that $L$ is strictly stronger than $L^\natural$ ($L$-convex set are $L^\natural$-convex, and $L$-log-concave function are $L^\natural$-log-concave). For instance, $L^\natural(u, v)$ is, aside from trivial cases, strictly smaller than $L(u, v).$ However, for elements of $\Z^n$ ending in $0$, we clearly have the following.
\begin{lemma}\label{lm:ends0lnat}
    For $a, b, c\in \Z^n$ ending in $0$, one has
    $a\in L(b, c)$ if and only if $a\in L^\natural(b, c).$
\end{lemma}

We will use the following theorem from \cite{speyer2026log} which says that marginals of $L$-convex sets are $L$-log-concave.
\begin{thm}[{\cite[Theorem 5.1]{speyer2026log}}]\label{thm:marginall}
    Let $K\sbs \Z^{m+n}$ be an $L$-convex set. For $x\in \Z^m$, define
    \[E_K(x) := \abs*{\{y\in \Z^n\colon (x,y)\in K\}}.\]
    Then $E_K$ is $L$-log-concave.
\end{thm}

\subsection{$L$-log-concavity of $(\lambda, \mu)\mapsto s_{\lambda/\mu}.$}

Speyer \cite{speyer2026log} proved the following.
\begin{thm}[{\cite[Main Theorem]{speyer2026log}}]\label{thm:speyermain}
    Let $\lambda, \mu, \ti\lambda, \ti\mu, \nu\in \Z^n$ be partitions such that $\lambda + \mu = \ti\lambda + \ti\mu$ and $\ti\lambda, \ti\mu\in L(\lambda, \mu).$
     Then $c_{\lambda,\mu}^\nu\le c_{\ti\lambda,\ti\mu}^{\nu}.$
\end{thm}
\begin{cor}\label{cor:lpp}
    Let $\lambda, \mu, \ti\lambda, \ti\mu, \in \Z^n$ be partitions  such that $\lambda + \mu = \ti\lambda + \ti\mu$ and 
    $\ti\lambda, \ti\mu\in \L^\natural(\lambda, \mu).$ Then $s_{\lambda}s_{\mu}\le_{s}s_{\ti\lambda}s_{\ti\mu}.$
\end{cor}
\begin{remark}
    While Theorem~\ref{thm:speyermain} allows $\ti\lambda, \ti\mu$ to be in $L(\lambda, \mu)$, in order to show $s_{\lambda}s_{\mu}\le_{s}s_{\ti\lambda}s_{\ti\mu}$, one needs $c_{\lambda,\mu}^\nu\le c_{\ti\lambda,\ti\mu}^{\nu}$ for all $\nu$, possibly of for $\nu$ of length longer than $n.$ To deal with this, one needs to pad partitions with $0's$ at the end. The condition of being in $L(\lambda, \mu)$ after padding by $0'$s is exactly the condition of   $L^\natural(\lambda, \mu).$ In fact, Corollary~\ref{cor:lpp} is false if $L^\natural(\lambda, \mu)$ is replaced by $L(\lambda, \mu).$ 
\end{remark}
As observed by Speyer \cite{speyer2026log}, the Lam--Postnikov--Pylyavskyy conjecture should be seen as saying that the function $\lambda\mapsto s_\lambda$ is  $L^\natural$-log-concave. Here we prove a generalization of the Lam--Postnikov--Pylyavskyy conjecture to skew Schur function.
\begin{thm}
    Let $\lambda,\mu, \nu,\rho,  \ti\lambda,\ti\mu,\ti\nu,\ti\rho\in \Z^n$ be  partitions such that  $(\lambda, \mu) + (\nu, \rho) = (\lambda', \mu') + (\mu', \rho')$ and $(\lambda', \mu'), (\nu', \rho')\in L((\lambda, \mu), (\nu, \rho))$ as elements of $\Z^{2n}.$
    Then 
    \[s_{\lambda/\mu}s_{\nu/\rho}\le_{s}s_{\ti\lambda/\ti\mu}s_{\ti\nu/\ti\rho}.\]
    \begin{proof}
        By Theorem \ref{thm:skew_skep}, it suffices to show that \[|\SkewSkep_{\lambda/\mu,\nu/\rho}^\kappa|\le |\SkewSkep_{\ti\lambda/\ti\mu,\ti\nu/\ti\rho}^\kappa|\]
        for all $\kappa.$ Following \cite{speyer2026log}, let $\Delta_{2n}^+ := \{(i, j)\in \Delta_{2n}\mid i+j = 0\mod 2\},$ and for each $g^+\colon \Delta_{2n}^+\to \Z,$ let $\SkewSkep_{\lambda/\mu,\nu/\rho}^\kappa(g^+)$ be the set of skew skep in $\SkewSkep_{\lambda/\mu,\nu/\rho}^\kappa$ whose restriction to $\Delta_{2n}^+$ is $g^+.$ We shall show the following stronger statement:
        \begin{equation}\abs*{\SkewSkep_{\lambda/\mu,\nu/\rho}^\kappa(g^+)}\le \abs*{\SkewSkep_{\ti\lambda/\ti\mu,\ti\nu/\ti\rho}^\kappa(g^+)}\label{eq:skskineq}\end{equation}
        for all $g^+.$ 

        Note that $ K:= \SkewSkep_{*, *}^\kappa(g^+)$, the set of all skew skeps $g$ whose restriction to $\Delta_{2n}^+$ is $g^+$ and $g^{\nwarrow} = \kappa$, is $L$-convex. Indeed, with the entries in even positions filled in, the inequalities in remaining entries are all of the form $g_{ij}-g_{k\ell}\le c_{ijk\ell}$ for various $i, j, k, \ell,$ and the condition $\kappa = g^{\nwarrow}$ is either always or never satisfied, as it depends only on $g^+.$ Let $(\alpha, \beta) := (\lambda, \mu) + (\nu, \rho) = (\ti\lambda, \ti\mu) + (\ti\nu, \ti\rho) \in \Z^{2n}$, and let \[\pi := (0, \alpha_1, \alpha_1 + \alpha_2, \dots, \alpha_1+\dots + \alpha_{n-1},0,\beta_1, \beta_1 + \beta_2, \dots, \beta_1+\dots + \beta_{n-1}).\]
        By considering the projection  map which remembers only the boundary entries that are not in $\Delta^+_{2n}$, 
        one sees that \[\abs*{\SkewSkep_{x/y, (\alpha-x)/(\beta-y)}^\kappa(g^+)} = E_K((x, y) + \pi).\]

        By Theorem~\ref{thm:marginall} and the fact that translations of $L$-log-concave function are log-concave, we obtain that the function 
        \[(x,y)\mapsto \abs*{\SkewSkep_{x/y, (\alpha-x)/(\beta-y)}^\kappa(g^+)}\]

        is $L$-log-concave. Thus we have 
        \begin{equation}\abs*{\SkewSkep_{\lambda/\mu,\nu/\rho}^\kappa(g^+)} \cdot \abs*{\SkewSkep_{\nu/\rho, \lambda/\mu}^\kappa(g^+)}\le \abs*{\SkewSkep_{\ti\lambda/\ti\mu,\ti\nu/\ti\rho}^\kappa(g^+)} \cdot \abs*{\SkewSkep_{\ti\nu/\ti\rho, \ti\lambda/\ti\mu}^\kappa(g^+)}.\label{eq:skskeqsq}\end{equation}
        
        By Theorem~\ref{thm:skew-hive-skep}, we have $\abs*{\SkewSkep_{\lambda/\mu,\nu/\rho}^\kappa(g^+)}= \abs*{\SkewSkep_{\nu/\rho, \lambda/\mu}^\kappa(g^+)}$ and \\$\abs*{\SkewSkep_{\ti\lambda/\ti\mu,\ti\nu/\ti\rho}^\kappa(g^+)}=\abs*{\SkewSkep_{\ti\nu/\ti\rho, \ti\lambda/\ti\mu}^\kappa(g^+)}.$ Thus taking the squareroot of \eqref{eq:skskeqsq} gives \eqref{eq:skskineq}.
    \end{proof}
\end{thm}

Finally, we recover some of Lam--Postnikov--Pylyavskyy's results. Recall that given two partitions $\lambda$ and $\mu$, we define $\lambda+\mu = (\lambda_1 +\mu_1,\lambda_2 + \mu_2,\ldots)$ and $\dfrac{\lambda}{2} = \left(\dfrac{\lambda_1}{2},\dfrac{\lambda_2}{2},\ldots\right)$. We define $\lambda \vee \mu = (\max(\lambda_1,\mu_1), \max(\lambda_2,\mu_2),\ldots)$ and $\lambda \wedge \mu = (\min(\lambda_1,\mu_1), \min(\lambda_2,\mu_2),\ldots)$. We extend the notation to skew shapes by $(\lambda/\mu)\vee(\nu/\rho) = (\lambda\vee\nu)/(\mu\vee\rho)$ and $(\lambda/\mu)\wedge(\nu/\rho) = (\lambda\wedge\nu)/(\mu\wedge\rho)$. Finally, let $\lambda\cup \mu = \nu$, then we define $\operatorname{sort}_1(\lambda,\mu) = (\nu_1,\nu_3,\nu_5,\ldots)$ and $\operatorname{sort}_2(\lambda,\mu) = (\nu_2,\nu_4,\nu_6,\ldots)$.

\begin{cor} Let $\lambda/\mu$ and $\nu/\rho$ be skew shapes. Then we have
    \begin{enumerate}[(a)]
        \item
        {\cite[Theorem 5]{LPP07}} $s_{(\lambda/\mu)\wedge (\nu/\rho)}s_{(\lambda/\mu)\vee (\nu/\rho)}\ge_s s_{\lambda/\mu}s_{\nu/\rho}.$
        \item {\cite[Theorem 12]{LPP07}} $s_{\floor*{\frac{\lambda+\nu}2}/\floor*{\frac{\mu+\rho}2}}s_{\ceil*{\frac{\lambda+\nu}2}/\ceil*{\frac{\mu+\rho}2}}\ge_s s_{\lambda/\mu}s_{\nu/\rho}.$ 
        \item {\cite[Corollary 14]{LPP07}} $s_{\sort_1(\lambda, \nu)/\sort_1(\mu, \rho)}s_{\sort_2(\lambda, \nu)/\sort_2(\mu, \rho)}\ge_s s_{\lambda/\mu}s_{\nu/\rho}.$
    \end{enumerate}
    \begin{proof}
        (a) and (b) are direct corollary. (c) follows from (b) by conjugation.
    \end{proof}
\end{cor}

\section{Applications}

\subsection{Newell-Littlewood numbers and  Koike-Terada log-concavity}

    Now we take a little detour to prove Theorem \ref{thm:new-lit}. Recall that
    \[ s_{[\mu]}s_{[\nu]} = \sum_\lambda N_{\mu,\nu,\lambda}s_{[\lambda]}, \]
    so by Lemma \ref{lm:ends0lnat} it suffices to prove the following theorem.
    
    \begin{thm}\label{thm:nwnumber}
        Let $\mu, \nu, \ti\mu, \ti\nu$ be partitions, such that $\ti\mu, \ti\nu\in L^\natural(\mu, \nu)$ and $\mu + \nu = \ti\mu+\ti\nu.$ Then $N_{\mu, \nu, \lambda}\le N_{\mu', \nu', \lambda}$ for all $\lambda$.
    \end{thm}

    Recall also that
    \[ N_{\mu,\nu,\lambda} = \sum_\alpha\langle s_{\mu/\alpha}s_{\nu/\alpha},s_{\lambda}\rangle, \]
    so Theorem \ref{thm:nwnumber} follows from the following lemma combined with Theorem \ref{thm:skew-schur-log}.

    \begin{lemma}
        Let $\mu, \ti\mu, \nu, \alpha\in \Z^n$. We have $\ti\mu\in L^{\natural}(\mu, \nu)$ if and only if $(\ti\mu,\alpha)\in  L((\mu,\alpha), (\nu,\alpha))$.
    \end{lemma}

    \begin{proof}
        We use the characterization $\ti\mu\in L^\natural(\mu, \nu)$ if and only if $\ti\mu_i - \ti\mu_j$ is between $\mu_i - \mu_j$ and $\nu_i - \nu_j$ and $\ti \mu_i$ is between $\mu_i$ and $\nu_i$ for all $i,j$; while $(\ti\mu,\alpha)\in  L((\mu,\alpha), (\nu,\alpha))$ holds if and only if $\ti\mu_i - \ti\mu_j$ is between $\mu_i - \mu_j$ and $\nu_i - \nu_j$, and $\ti \mu_i - \alpha_j$ is between $\mu_i-\alpha_j$ and $\nu_i-\alpha_j$ for all $i,j$. These are clearly equivalent.
    \end{proof}

\subsection{Shadow skew Schur functions}
Now we explain how to modify the proof of Theorem~\ref{thm:skew-schur-log} to obtain Theorem~\ref{thm:shadow-skew-schur-log}. 
\begin{proof}[Proof of Theorem~\ref{thm:shadow-skew-schur-log}]
    Observe that the coefficient of $s_\kappa$ in $S^{(k, \ell)}_{\lambda/\mu}S^{(k, \ell)}_{\nu/\rho}$ is counted by the number of skew skep $g$ such that $g^{\uparrow} = (\lambda'_1, \nu_1', \dots, \lambda_{k}', \nu'_{k}), g^{\rightarrow} = (\mu'_1, \rho_1', \dots, \mu_{k}', \rho'_{k}),$ and $g^{\nwarrow} = \kappa$ for some $\lambda' , \mu', \nu', \rho'$ such that  $\lambda'_i = \lambda_i, \mu'_i = \mu_i, \nu'_i = \nu_i, \rho'_i = \rho_i$ for $i\le \ell.$ These are exactly skew skeps $g$ with $g^{\nwarrow} = \kappa$ such that the first $2\ell$ entries of $g^{\uparrow}$ is $(\lambda_1, \nu_1, \dots, \lambda_\ell, \nu_\ell)$ and the first $2\ell$ entries of $g^{\rightarrow}$ is $(\mu_1, \rho_1, \dots, \mu_\ell, \rho_\ell).$ Let $\SkewSkep_{\lambda/\mu, \nu/\rho}^{\kappa, (k, \ell)}$ denotes this set. 
    As in the proof of Theorem~\ref{thm:skew-schur-log}, for a function $g^+\colon \Delta_{2n}^+\to \Z,$ let $\SkewSkep_{\lambda/\mu,\nu/\rho}^{\kappa, (k, \ell)}(g^+)$ be the set of skew skep in $\SkewSkep_{\lambda/\mu,\nu/\rho}^{\kappa, (k, \ell)}$ whose restriction to $\Delta_{2n}^+$ is $g^+.$ 
    It suffices to show that 
    \[\abs*{\SkewSkep_{\lambda/\mu,\nu/\rho}^{\kappa, (k, \ell)}(g^+)}\le \abs*{\SkewSkep_{\ti\lambda/\ti\mu,\ti\nu/\ti\rho}^{\kappa, (k, \ell)}(g^+)}\]
    for all $g^+.$ By applying to Theorem~\ref{thm:marginall} the projection map that remembers $g_{0,2i-1}$ and $g_{2i-1, 0}$ for $1\le i\le \ell$, we obtain that the function 
    \[(x,y)\mapsto \abs*{\SkewSkep_{x/y,(\alpha-x)/(\beta-y)}^{\kappa, (k, \ell)}(g^+)}\]
    is $L$-log concave, where $(\alpha, \beta) := (\lambda, \mu) + (\nu, \rho) = (\ti\lambda, \ti\mu) + (\ti\nu, \ti\rho).$ The rest of the proof is verbatim to the proof of Theorem~\ref{thm:skew-schur-log}. 
\end{proof}

\section{Phased skew hives and phased peelable tableaux}\label{sec:peelable}

\subsection{Peelable tableaux}

    First, we review Nguyen--Nguyen--Woodruff's peelable tableaux \cite{nguyen2025shuffle}.

    \begin{definition}
        A \textit{$(\lambda/\mu) \circledast (\nu/\rho)$-peelable tableau} is a SSYT of content $(\lambda_1-\mu_1,\nu_1-\rho_1,\lambda_2-\mu_2,\nu_2-\rho_2,\ldots, \lambda_n-\mu_n,\nu_n-\rho_n)$ such that
        \begin{itemize}
            \item if row $i$ and $i+1$ of $\lambda/\mu$ have $a_i$ columns in common, then $T$ has at least $a_i$ disjoint pairs of one $(2i-1)$-square and one $(2i+1)$-square such that the $(2i-1)$-square is Northeast of the $(2i+1)$-square, and
            \item if row $i$ and $i+1$ of $\nu/\rho$ have $b_i$ columns in common, then $T$ has at least $b_i$ disjoint pairs of one $2i$-square and one $(2i+2)$-square such that the $2i$-square is Northeast of the $(2i+2)$-square.
        \end{itemize}
    \end{definition}

    \begin{thm}[{\cite[Theorem 1.2]{nguyen2025shuffle}}]\label{thm:main-thm}
        The Littlewood--Richardson coefficient $c^{\kappa}_{\lambda/\mu, \nu/\rho}$ counts the number of $(\lambda/\mu) \circledast (\nu/\rho)$-peelable tableaux of shape $\kappa$.
    \end{thm}

    \begin{example}\label{ex:peelable1}
        Let $\lambda = (2,2), \mu = (1,0), \nu = (4,3), \rho = (2,0)$, and $\kappa = (4,3,1)$ as in Examples \ref{ex:skew-hives} and \ref{ex:skew-skeps}. There are three $(\lambda/\mu)\circledast(\nu/\rho)$-peelable tableaux of shape $\kappa$
        \[ \begin{ytableau} 
            \textcolor{red}{1} & \textcolor{blue}{2} & \textcolor{blue}{2} & \textcolor{red}{3} \\ 
            \textcolor{red}{3} & \textcolor{blue}{4} & \textcolor{blue}{4} \\
            \textcolor{blue}{4}
            \end{ytableau}, \quad\quad\quad \begin{ytableau} 
            \textcolor{red}{1} & \textcolor{blue}{2} & \textcolor{blue}{2} & \textcolor{blue}{4} \\ 
            \textcolor{red}{3} & \textcolor{red}{3} & \textcolor{blue}{4} \\
            \textcolor{blue}{4}
            \end{ytableau}, \quad\quad\quad \begin{ytableau} 
            \textcolor{red}{1} & \textcolor{blue}{2} & \textcolor{red}{3} & \textcolor{blue}{4} \\ 
            \textcolor{blue}{2} & \textcolor{blue}{4} & \textcolor{blue}{4} \\
            \textcolor{red}{3}
            \end{ytableau}. \]
        Here, the rules for the red numbers come from $\lambda/\mu$ and those for the blue numbers come from $\nu/\rho$.
    \end{example}

    There is an easy shape-preserving bijection between $(\lambda/\mu)\circledast(\nu/\rho)$-peelable tableaux and $(\nu/\rho)\circledast(\lambda/\mu)$-peelable tableaux using Bender--Knuth involutions.

    \begin{definition}\label{defn:BK_CST}
    The {\it Bender--Knuth involution} $\BK_i$ for $1 \le i \le n-1$ is an involution of the set of SSYT of shape $\lambda$ that sends a SSYT $T$ to the SSYT obtained from the following procedure:
    \begin{enumerate}
        \item Let $S$ be the skew tableau obtained by taking only the squares of $T$ with entry equal $i$ and $i+1$;
        \item Each row of $S$ contains  $a$ entries $i$ that are directly above an $i+1$, $b$ entries $i$ that are alone in their columns, $c$ entries $i+1$ that are alone in their columns, and $d$ entries $i+1$ that are directly below an $i$ for some $a,b,c,d \geq 0$;
        \item Construct a skew tableau $S'$ by swapping $b$ and $c$ in each row of $S$;
        \item Define $\BK_i(T)$ to be the tableau obtained by replacing $S$ with $S'$ in $T$.
    \end{enumerate}
    If $\alpha = (\alpha_1,\ldots,\alpha_i,\alpha_{i+1},\ldots,\alpha_\ell)$ is the content of $T$, then $\alpha = (\alpha_1,\ldots,\alpha_{i+1},\alpha_i,\ldots,\alpha_\ell)$ is the content of $\BK_i(T)$.
    \end{definition}

    \begin{bijection}[{\cite[Theorem 1.4]{nguyen2025shuffle}}]\label{bij:peelable}
        Given partitions $\lambda,\mu,\nu,\rho$ of length $n$, and a $(\lambda/\mu)\circledast(\nu/\rho)$-peelable tableau $T$, then
        \[ T' = \BK_1\circ\BK_3\circ\cdots\circ\BK_{2n - 1} (T) \]
        is a $(\nu/\rho)\circledast(\lambda/\mu)$-peelable tableau.
    \end{bijection}

    \begin{example}\label{ex:peelable2}
        Applying $\BK_1\circ\BK_3$ to the two peelable tableaux in Example \ref{ex:peelable1}, we get
        \[ \begin{ytableau} 
            \textcolor{blue}{1} & \textcolor{blue}{1} & \textcolor{red}{2} & \textcolor{red}{4} \\ 
            \textcolor{blue}{3} & \textcolor{blue}{3} & \textcolor{blue}{3} \\
            \textcolor{red}{4}
            \end{ytableau}, \quad\quad\quad \begin{ytableau} 
            \textcolor{blue}{1} & \textcolor{blue}{1} & \textcolor{red}{2} & \textcolor{blue}{3} \\ 
            \textcolor{blue}{3} & \textcolor{blue}{3} & \textcolor{red}{4} \\
            \textcolor{red}{4}
            \end{ytableau}, \quad\quad\quad \begin{ytableau} 
            \textcolor{blue}{1} & \textcolor{blue}{1} & \textcolor{blue}{3} & \textcolor{blue}{3} \\ 
            \textcolor{red}{2} & \textcolor{blue}{3} & \textcolor{red}{4} \\
            \textcolor{red}{4}
            \end{ytableau}. \]
        We encourage the readers to verify that they are all $(\nu/\rho)\circledast(\lambda/\mu)$-peelable, where $\lambda,\mu,\nu,\rho$ are as in Example \ref{ex:peelable1}
    \end{example}

    As a consequence of Bijection \ref{bij:peelable}, we actually have $2n$ different Littlewood--Richardson rules.

\subsection{Phased interlacing paths and phased peelable tableaux}\label{subsec:path-peelable}

    \begin{definition}
        A \textit{path} $P$ of length $n$ is a path from $(n,0)$ to $(0,n)$ that consists of only $N$ steps $(0,1)$ and $W$ steps $(-1,0)$. For $1\leq k\leq n$, the \textit{$k$-phased interlacing path}, denoted $P_k$ is the path $W^kNWNW\ldots NWN^k$. The \textit{$-k$-phased interlacing path}, denoted $P_{-k}$ is the path $N^kWNWN\ldots WNW^k$

        Given a path $P$, let $N_{P,i}$ be the index of the $i$th $N$ step, and let $W_{P,i}$ be the index of the $i$th $W$ step. For example, $N_{P_1,i} = 2i$ and $W_{P_1,i} = 2i-1$.
    \end{definition}

    \begin{definition}
        A \textit{$k$-phased $(\lambda/\mu) \circledast_k (\nu/\rho)$-peelable tableau} is a SSYT of content $(\lambda_1-\mu_1,\ldots,\lambda_k-\mu_k,\nu_1-\rho_1,\lambda_{k+1}-\mu_{k+1},\nu_2-\rho_2,\ldots, \lambda_n-\mu_n,\nu_{n-k+1}-\rho_{n-k+1},\ldots,\nu_n-\rho_n)$ such that
        \begin{itemize}
            \item if row $i$ and $i+1$ of $\lambda/\mu$ have $a_i$ columns in common, then $T$ has at least $a_i$ disjoint pairs of one $W_{P_k,i}$-square and one $W_{P_k,i+1}$-square such that the $W_{P_k,i}$-square is Northeast of the $W_{P_k,i+1}$-square, and
            \item if row $i$ and $i+1$ of $\nu/\rho$ have $b_i$ columns in common, then $T$ has at least $b_i$ disjoint pairs of one $N_{P_k,i}$-square and one $N_{P_k,i+1}$-square such that the $N_{P_k,i}$-square is Northeast of the $N_{P_k,i+1}$-square.
        \end{itemize}
    \end{definition}

    Repeated applications of Bijection \ref{bij:peelable} give the following $2n$ rules.

    \begin{cor}
        For $1\leq |k|\leq n$, the Littlewood--Richardson coefficient $c^{\kappa}_{\lambda/\mu, \nu/\rho}$ counts the number of $(\lambda/\mu) \circledast_k (\nu/\rho)$-peelable tableaux of shape $\kappa$.
    \end{cor}

    \begin{example}\label{ex:phased-peelable}
        Let $\lambda,\mu,\nu,\rho$, and $\kappa$ be as in Example \ref{ex:peelable1}. The three SSYTs in Example \ref{ex:peelable1} are $(\lambda/\mu) \circledast_1 (\nu/\rho)$-peelable, and the three SSYTs in Example \ref{ex:peelable2} are $(\lambda/\mu) \circledast_{-1} (\nu/\rho)$-peelable. Applying $\BK_2$ to the SSYTs in Example \ref{ex:peelable1}, we get
        \[ \begin{ytableau} 
            \textcolor{red}{1} & \textcolor{red}{2} & \textcolor{blue}{3} & \textcolor{blue}{3} \\ 
            \textcolor{red}{2} & \textcolor{blue}{4} & \textcolor{blue}{4} \\
            \textcolor{blue}{4}
            \end{ytableau}, \quad\quad\quad \begin{ytableau} 
            \textcolor{red}{1} & \textcolor{red}{2} & \textcolor{blue}{3} & \textcolor{blue}{4} \\ 
            \textcolor{red}{2} & \textcolor{blue}{3} & \textcolor{blue}{4} \\
            \textcolor{blue}{4}
            \end{ytableau}, \quad\quad\quad \begin{ytableau} 
            \textcolor{red}{1} & \textcolor{red}{2} & \textcolor{blue}{3} & \textcolor{blue}{4} \\ 
            \textcolor{red}{2} & \textcolor{blue}{4} & \textcolor{blue}{4} \\
            \textcolor{blue}{3}
            \end{ytableau}. \]
            These are $(\lambda/\mu) \circledast_2 (\nu/\rho)$-peelable.
    \end{example}

\subsection{Phased skew hives}

    \begin{definition}
        For $k\in \ZZ$, let the height of a diagonal $x+y = k$ be $k-n$. Given a path $P$, let $\hht_P(i)$ be the height of $P$ after the $i$th step. Here, we take $\hht_P(0) = 0$. For example, $\hht_{P_1}(i) = 0$ if $i$ is even and $\hht_{P_1}(i) = -1$ if $i$ is odd.

        Given the $k$-phased path $P_k$, the \textit{height array} of $P_k$ is the function $\Ht_{P_k}:\Delta_{2n}\rightarrow \ZZ$ defined by
        \[ \Ht_{P_k}(i,j) = \hht(2n-i-j). \]
        For example, for $n = 3$, the height array of $P_1$ is in Figure \ref{subfig:skew-skep-coor}, and the height array of $P_3$ is in Figure \ref{subfig:skew-hive-coor}. The height array of $P_2$ is
        \[ \begin{tikzcd}[sep=tiny]
            0 &&&&&& \\
            {-1} & 0 \\
            {-2} & {-1} & 0 \\
            {-1} & {-2} & {-1} & 0 \\
            {-2} & {-1} & {-2} & {-1} & 0 \\
            {-1} & {-2} & {-1} & {-2} & {-1} & 0 \\
            0 & {-1} & {-2} & {-1} & {-2} & {-1} & 0
        \end{tikzcd} \]
    \end{definition}

    One can think of the height array as integral points on a section of $\Pi_n$. Then, one has a collection of rhombus inequalities on this section. These rhombus inequalities define $k$-phased skew hives.

    \begin{definition}
        A \textit{$k$-phased skew hive} is a function $h:\Delta_{2n}\rightarrow\ZZ$ obeying the rhombus inequalities determined by $\Ht_{P_k}$.
        
        Let the \textit{$k$-shuffling} of $\lambda$ and $\nu$ be
        \[ \lambda\circledast_k\nu = (\lambda_1,\ldots, \lambda_k,\nu_1,\lambda_{k+1},\nu_2,\ldots,\lambda_n,\nu_{n-k+1},\ldots,\nu_n) \]
        Let $k$-$\SkewHive_{\lambda/\mu,\nu/\rho}^\kappa$ be the set of $k$-phased skew hive with $h_{00} = 0$, $h^\uparrow = \lambda\circledast_k\nu$, $h^\rightarrow = \mu\circledast_k\rho$, and $h^\nwarrow = \kappa$. In particular, $n$-$\SkewHive_{\lambda/\mu,\nu/\rho}^\kappa = \SkewHive_{\lambda/\mu,\nu/\rho}^\kappa$, and $1$-$\SkewHive_{\lambda/\mu,\nu/\rho}^\kappa = \SkewSkep_{\lambda/\mu,\nu/\rho}^\kappa$.
    \end{definition}

    \begin{thm}\label{thm:phased-skew-hives}
        Given partitions $\lambda,\mu,\nu,\rho$ of length $n$, $\kappa$ of length $2n$, and $1\leq |k|\leq n$, then
        \[ c_{\lambda/\mu, \nu/\rho}^\kappa = |\text{$k$-}\SkewHive_{\lambda/\mu,\nu/\rho}^\kappa|. \]
    \end{thm}

    One can prove Theorem \ref{thm:phased-skew-hives} using the same technique as Theorem \ref{thm:skew-hive-skep}. However, we will give another proof by showing a bijection between $k$-phased skew hives and $k$-phased peelable tableaux.

    \begin{bijection}\label{bij:skew-hive-peelable}
        Let $\lambda,\mu,\nu,\rho$ be partitions with $n$ parts. For $1\leq |k| \leq n$, let $T$ be a $(\lambda/\mu)\circledast_k (\nu/\rho)$-peelable tableau,
        \begin{itemize}
            \item let $\{x_{i,j}~|~0\leq i\leq 2n, 1\leq j\leq 2n\}$ be the Gelfand-Tsetlin pattern, where $(x_{i,1},x_{i,2},\ldots,x_{i,n})$ is the shape of of the subtableaux containing the entries $0,\ldots,i$,
            \item let $\{y_{i,j}~|~0\leq i, j\leq 2n\}$ be the row sum of $\{x_{i,j}\}$, where $y_{i,0} = 0$ and $y_{i,j} = \sum_{a=1}^j x_{i,a}$,
            \item let $\mu' = \mu\circledast_{k}\rho$, add $\sum_{a = 1}^i \mu'_a$ to each $y_{i,j}$,
            \item the subpattern $\{y_{i,j}~|~i\geq j\}$ forms a skew hive in $k$-$\SkewHive_{\lambda/\mu,\nu/\rho}^\kappa$.
        \end{itemize}
    \end{bijection}

    \begin{proof}

        It suffices to prove for $1<k<n$. First we prove for $k = n$.
        
        Recall that we have Bijection \ref{bij:hive-skewhive} from $\Hive_{\mu\ast_{\nu_1}\rho,\kappa}^{\lambda\ast_{\nu_1}\nu}$ to $\SkewHive_{\lambda/\mu,\nu/\rho}^\kappa$ by subtracting each diagonal by some multiple of $\nu_1$. Furthermore, we have Bijection \ref{bij:yama-hive} from Yamanouchi tableaux of shape $(\lambda\ast_{\nu_1}\nu) /(\mu\ast_{\nu_1}\rho)$ to hives by taking row sums of the Gelfand--Tsetlin pattern.
        
        It was shown in \cite{remmel1984multiplying} that there is a bijection between $(\lambda/\mu)\circledast_n (\nu/\rho)$-peelable tableaux and Yamanouchi tableaux of shape $(\lambda\ast_{\nu_1}\nu) /(\mu\ast_{\nu_1}\rho)$. Let $T$ be a $(\lambda/\mu)\circledast_n (\nu/\rho)$-peelable tableau, we obtain a Yamanouchi tableau $T'$ of shape $(\lambda\ast_{\nu_1}\nu) /(\mu\ast_{\nu_1}\rho)$ by
        \[ \text{\# $i$ in row $j$ of $T$} = \text{\# $j$ in row $i$ of $T'$}. \]
        Let $y$ and $y'$ be the row sum arrays of $T$ and $T'$, respectively. The effect of this bijection on the row sum array is reflection $y_{i,j}\leftrightarrow y'_{j,i}$. Thus, subtracting multiples of $\nu_1$ from rows of $y$ is the same as subtracting from diagonals of $y'$, which gives a skew hive.

        Now we proceed by induction. Suppose the bijection is true for $k$, we prove for $k-1$. To obtain a $(k-1)$-phased skew hive from a $k$-phased skew hive, we do the octahedron recurrence at diagonals $i+j = k,k+2,\ldots,2n-k$. Meanwhile, to obtain a $(\lambda/\mu)\circledast_{k-1} (\nu/\rho)$-peelable tableau from a $(\lambda/\mu)\circledast_k (\nu/\rho)$-peelable tableau, we do $\BK_k\circ \BK_{k+2} \circ \cdots \circ\BK_{2n-k}$. It is not difficult to check that Bender--Knuth involutions coincide with the octahedron recurrence after the procedure.
    \end{proof}

    \begin{example}\label{ex:peelable-GT}
        Let $\lambda/\mu = (6,6,2)/(3,1,0)$ and $\nu/\rho = (4,4,3)/(2,2,0)$. Let $T$ be the following $(\lambda/\mu)\circledast_3 (\nu/\rho)$-peelable tableau
        \[ T = \begin{ytableau} 
            \textcolor{red}{1} & \textcolor{red}{1} & \textcolor{red}{1} & \textcolor{red}{2} & \textcolor{red}{2} & \textcolor{blue}{4} & \textcolor{blue}{6} \\ 
            \textcolor{red}{2} & \textcolor{red}{2} & \textcolor{red}{3} & \textcolor{blue}{5} \\
            \textcolor{red}{3} & \textcolor{blue}{4} & \textcolor{blue}{6} \\
            \textcolor{blue}{5} \\
            \textcolor{blue}{6}
            \end{ytableau}. \]
        The Gelfand--Tsetlin pattern is
        \[ \begin{matrix}
            7 && 4 && 3 && 1 && 1 && 0 \\
            & 6 && 4 && 2 && 1 && 0 && 0 \\
            && 6 && 3 && 2 && 0 && 0 && 0 \\
            &&& 5 && 3 && 1 && 0 && 0 && 0 \\
            &&&& 5 && 2 && 0 && 0 && 0 && 0 \\
            &&&&& 3 && 0 && 0 && 0 && 0 && 0 \\
            &&&&&& 0 && 0 && 0 && 0 && 0 && 0
        \end{matrix}. \]
        The row sum array is
        \[ \begin{matrix}
            0 && 7 && 11 && 14 && 15 && 16 && 16 \\
            & 0 && 6 && 10 && 12 && 13 && 13 && 13 \\
            && 0 && 6 && 9 && 11 && 11 && 11 && 11 \\
            &&& 0 && 5 && 8 && 9 && 9 && 9 && 9 \\
            &&&& 0 && 5 && 7 && 7 && 7 && 7 && 7 \\
            &&&&& 0 && 3 && 3 && 3 && 3 && 3 && 3 \\
            &&&&&& 0 && 0 && 0 && 0 && 0 && 0 && 0
        \end{matrix}. \]
        We have $\mu' = (\mu,\rho) = (3,1,0,2,2,0)$, so adding $\sum_{k_1}^i\mu'_k$ to each row $i$, we have
        \[ y = \begin{matrix}
            \textcolor{red}{8} && \textcolor{red}{15} && \textcolor{red}{19} && \textcolor{red}{22} && \textcolor{red}{23} && \textcolor{red}{24} && \textcolor{red}{24} \\
            & \textcolor{red}{8} && \textcolor{red}{14} && \textcolor{red}{18} && \textcolor{red}{20} && \textcolor{red}{21} && \textcolor{red}{21} && 21 \\
            && \textcolor{red}{6} && \textcolor{red}{12} && \textcolor{red}{15} && \textcolor{red}{17} && \textcolor{red}{17} && 17 && 17 \\
            &&& \textcolor{red}{4} && \textcolor{red}{9} && \textcolor{red}{12} && \textcolor{red}{13} && 13 && 13 && 13 \\
            &&&& \textcolor{red}{4} && \textcolor{red}{9} && \textcolor{red}{11} && 11 && 11 && 11 && 11 \\
            &&&&& \textcolor{red}{3} && \textcolor{red}{6} && 6 && 6 && 6 && 6 && 6 \\
            &&&&&& \textcolor{red}{0} && 0 && 0 && 0 && 0 && 0 && 0
        \end{matrix}. \]
        The red entries form the skew hive
        \[ h = \begin{matrix}
            24 \\
            21 & 24 \\
            17 & 21 & 23 \\
            13 & 17 & 20 & 22 \\
            11 & 12 & 15 & 18 & 19 \\
            6 & 9 & 9 & 12 & 14 & 15 \\
            0 & 3 & 4 & 4 & 6 & 8 & 8
        \end{matrix}. \]
        
        On the other hand, the above peelable tableau bijects to the following Yamanouchi tableau
        \[ T' = \begin{ytableau} 
            \none & \none & \none & \none & \none & \none & \none & 1 & 1 & 1 \\
            \none & \none & \none & \none & \none & 1 & 1 & 2 & 2 \\
            \none & \none & \none & \none & 2 & 3 \\
            \none & \none & 1 & 3 \\
            \none & \none & 2 & 4 \\
            1 & 3 & 5
            \end{ytableau}. \]
        The row sum array is
        \[ \begin{matrix}
            0 && 10 && 19 && 25 && 29 && 33 && 36 \\
            & 0 && 10 && 19 && 25 && 29 && 33 && 36 \\
            && 0 && 10 && 19 && 25 && 29 && 33 && 35 \\
            &&& 0 && 10 && 19 && 25 && 29 && 32 && 34 \\
            &&&& 0 && 10 && 19 && 24 && 27 && 30 && 31 \\
            &&&&& 0 && 10 && 17 && 21 && 24 && 26 && 27 \\
            &&&&&& 0 && 7 && 12 && 16 && 18 && 20 && 20
        \end{matrix}. \]
        Subtracting multiples of $\nu_1 = 4$ from diagonals, we have
        \[ y' = \begin{matrix}
            0 && 6 && 11 && 13 && 17 && 21 && 24 \\
            & 0 && 6 && 11 && 13 && 17 && 21 && 24 \\
            && 0 && 6 && 11 && 13 && 17 && 21 && 23 \\
            &&& 0 && 6 && 11 && 13 && 17 && 20 && 22 \\
            &&&& 0 && 6 && 11 && 12 && 15 && 16 && 19 \\
            &&&&& 0 && 6 && 9 && 9 && 12 && 14 && 15 \\
            &&&&&& 0 && 3 && 4 && 4 && 6 && 8 && 8
        \end{matrix}. \]
        This is the reflection of $y$.

        Now to obtain a $2$-phased skew hive from $h$, we do the octahedron recurrence along the diagonals $i+j = 3$, which gives
        \[ h' = \begin{matrix}
            24 \\
            21 & 24 \\
            17 & 21 & 23 \\
            \textcolor{red}{15} & 17 & 20 & 22 \\
            11 & \textcolor{red}{14} & 15 & 18 & 19 \\
            6 & 9 & \textcolor{red}{12} & 12 & 14 & 15 \\
            0 & 3 & 4 & \textcolor{red}{6} & 6 & 8 & 8
        \end{matrix}. \]
        Applying $\BK_3$ to $T$, we get
        \[ \begin{ytableau} 
            \textcolor{red}{1} & \textcolor{red}{1} & \textcolor{red}{1} & \textcolor{red}{2} & \textcolor{red}{2} & \textcolor{blue}{3} & \textcolor{blue}{6} \\ 
            \textcolor{red}{2} & \textcolor{red}{2} & \textcolor{red}{4} & \textcolor{blue}{5} \\
            \textcolor{blue}{3} & \textcolor{red}{4} & \textcolor{blue}{6} \\
            \textcolor{blue}{5} \\
            \textcolor{blue}{6}
            \end{ytableau}. \]
        The row sum array of this peelable tableau only differs from $y$ on row $3$
        \[ \begin{matrix}
            8 && 15 && 19 && 22 && 23 && 24 && 24 \\
            & 8 && 14 && 18 && 20 && 21 && 21 && 21 \\
            && 6 && 12 && 15 && 17 && 17 && 17 && 17 \\
            &&& \textcolor{red}{6} && \textcolor{red}{12} && \textcolor{red}{14} && \textcolor{red}{15} && \textcolor{red}{15} && \textcolor{red}{15} && \textcolor{red}{15} \\
            &&&& 4 && 9 && 11 && 11 && 11 && 11 && 11 \\
            &&&&& 3 && 6 && 6 && 6 && 6 && 6 && 6 \\
            &&&&&& 0 && 0 && 0 && 0 && 0 && 0 && 0
        \end{matrix}. \]
        The subpattern $\{y_{i,j}~|~i\geq j\}$ forms $h'$.
    \end{example}

\section{Final remarks}

    It would be interesting to further explore the connection between our result and Chan-Chen-Pak-Soskin's results in \cite{chan2026correlation}. In particular, they resolved Mihalcea’s conjecture on log-supermodularity of stable Grothendieck polynomials. It is natural to ask for an analogue for skew stable Grothendieck polynomials.

    For every root system $R\subset V$, Lam--Postnikov \cite{lam2024polypositroids} defined \textit{$R$-membranes}, which are triangulated $2$-dimensional surfaces in $V$ homeomorphic to wedges of disks such that every edge of every triangle in them is a parallel translation of a root in $R$. In type A, membranes are in bijection with plabic graphs with faces labelled by integer vectors, and they showed that local moves of plabic graphs correspond to octahedron and tetrahedron moves of membranes. See also \cite{farber2016arrangements}. It would be interesting to use the technology of membranes and plabic graphs to extend the peelable tableaux rules in Section \ref{subsec:path-peelable} to arbitrary paths, as well as higher dimension paths for product of more than 2 skew Schur functions.

    Finally, in \cite{gleizer2000littlewood}, Gleizer--Postnikov studied the Littlewood--Richardson rule in the language of \textit{web functions}, which are closely related to Knutson--Tao's honeycombs and Berenstein--Zelevinsky patterns. They used it to show the duality of the Littlewood--Richardson coefficients under conjugation of partitions. It would be interesting to study the skew version of their model.

\bibliography{bibliography}

@article {GOY21,
    AUTHOR = {Gao, Shiliang and Orelowitz, Gidon and Yong, Alexander},
     TITLE = {Newell-{L}ittlewood numbers},
   JOURNAL = {Trans. Amer. Math. Soc.},
  FJOURNAL = {Transactions of the American Mathematical Society},
    VOLUME = {374},
      YEAR = {2021},
    NUMBER = {9},
     PAGES = {6331--6366},
      ISSN = {0002-9947,1088-6850},
   MRCLASS = {05E10},
  MRNUMBER = {4302162},
MRREVIEWER = {Himmet\ Can},
       DOI = {10.1090/tran/8375},
       URL = {https://doi.org/10.1090/tran/8375},
}

@article {LPP07,
    AUTHOR = {Lam, Thomas and Postnikov, Alexander and Pylyavskyy, Pavlo},
     TITLE = {Schur positivity and {S}chur log-concavity},
   JOURNAL = {Amer. J. Math.},
  FJOURNAL = {American Journal of Mathematics},
    VOLUME = {129},
      YEAR = {2007},
    NUMBER = {6},
     PAGES = {1611--1622},
      ISSN = {0002-9327,1080-6377},
   MRCLASS = {05E05},
  MRNUMBER = {2369890},
MRREVIEWER = {Riccardo\ Biagioli},
       DOI = {10.1353/ajm.2007.0045},
       URL = {https://doi.org/10.1353/ajm.2007.0045},
}

@misc{BHKKL25,
      title={Representation theory for polymatroids}, 
      author={Matthew Baker and June Huh and Donggyu Kim and Mario Kummer and Oliver Lorscheid},
      year={2025},
      eprint={2507.14718},
      archivePrefix={arXiv},
      primaryClass={math.CO},
      url={https://arxiv.org/abs/2507.14718}, 
}

@book {M03,
    AUTHOR = {Murota, Kazuo},
     TITLE = {Discrete convex analysis},
    SERIES = {SIAM Monographs on Discrete Mathematics and Applications},
 PUBLISHER = {Society for Industrial and Applied Mathematics (SIAM),
              Philadelphia, PA},
      YEAR = {2003},
     PAGES = {xxii+389},
      ISBN = {0-89871-540-7},
   MRCLASS = {90-02 (52-02 90C27 90C46 91B02)},
  MRNUMBER = {1997998},
MRREVIEWER = {Ulrich\ Faigle},
       DOI = {10.1137/1.9780898718508},
       URL = {https://doi.org/10.1137/1.9780898718508},
}

@article{remmel1984multiplying,
  title={Multiplying {S}chur functions},
  author={Remmel, Jeffrey B. and Whitney, Roger},
  journal={Journal of Algorithms},
  volume={5},
  number={4},
  pages={471--487},
  year={1984},
  publisher={Elsevier}
}

@article{knutson1999honeycomb,
  title={The honeycomb model of {$GL_n(\mathbb{C})$} tensor products {I}: {P}roof of the saturation conjecture},
  author={Knutson, Allen and Tao, Terence},
  journal={Journal of the American Mathematical Society},
  volume={12},
  number={4},
  pages={1055--1090},
  year={1999}
}

@article{knutson2003positive,
  title={A positive proof of the {L}ittlewood-{R}ichardson rule using the octahedron recurrence},
  author={Knutson, Allen and Tao, Terence and Woodward, Christopher T},
  journal={arXiv preprint math/0306274},
  year={2003}
}

@article{lam2007alcoved,
  title={Alcoved polytopes, {I}},
  author={Lam, Thomas and Postnikov, Alexander},
  journal={Discrete \& Computational Geometry},
  volume={38},
  pages={453--478},
  year={2007},
  publisher={Springer}
}

@article{lam2018alcoved,
  title={Alcoved polytopes {II}},
  author={Lam, Thomas and Postnikov, Alexander},
  journal={Lie Groups, Geometry, and Representation Theory: A Tribute to the Life and Work of Bertram Kostant},
  pages={253--272},
  year={2018},
  publisher={Springer}
}

@article{dobrovolska2007products,
  title={On products of {$\mathfrak{sl}_n$} characters and support containment},
  author={Dobrovolska, Galyna and Pylyavskyy, Pavlo},
  journal={Journal of Algebra},
  volume={316},
  number={2},
  pages={706--714},
  year={2007},
  publisher={Elsevier}
}

@article{nguyen2025temperley,
  title={Temperley--{L}ieb Crystals},
  author={Nguyen, Son and Pylyavskyy, Pavlo},
  journal={International Mathematics Research Notices},
  volume={2025},
  number={7},
  year={2025},
  publisher={Oxford University Press}
}

@phdthesis{pylyavskyy2007comparing,
  title={Comparing products of {S}chur functions and quasisymmetric functions},
  author={Pylyavskyy, Pavlo},
  year={2007},
  school={Massachusetts Institute of Technology}
}

@article{nguyen2025shuffle,
  title={Shuffle {T}ableaux, {L}ittlewood--{R}ichardson {C}oefficients, and {S}chur {L}og-{C}oncavity},
  author={Nguyen, Chau and Nguyen, Son and Woodruff, Dora},
  journal={arXiv preprint arXiv:2506.00349},
  year={2025}
}

@article{speyer2026log,
  title={L-log-concavity and a proof of the conjecture of {L}am, {P}ostnikov and {P}ylyavskyy},
  author={Speyer, David E},
  journal={arXiv preprint arXiv:2601.05007},
  year={2026}
}

@article{buch2000saturation,
  title={The saturation conjecture (after {A}. {K}nutson and {T}. {T}ao), with an appendix by {W}illiam {F}ulton},
  author={Buch, Anders Skovsted},
  journal={ENSEIGNEMENT MATHEMATIQUE},
  volume={46},
  number={1/2},
  pages={43--60},
  year={2000},
  publisher={SWETS \& ZEITLINGER}
}

@article{henriques2006octahedron,
  title={The octahedron recurrence and {$\mathfrak{gl}_n$} crystals},
  author={Henriques, Andr{\'e} and Kamnitzer, Joel},
  journal={Advances in Mathematics},
  volume={206},
  number={1},
  pages={211--249},
  year={2006},
  publisher={Elsevier}
}

@inproceedings{newell1951modification,
  title={Modification rules for the orthogonal and symplectic groups},
  author={Newell, Martin J},
  booktitle={Proceedings of the Royal Irish Academy. Section A: Mathematical and Physical Sciences},
  volume={54},
  pages={153--163},
  year={1951},
  organization={JSTOR}
}

@article{littlewood1958products,
  title={Products and plethysms of characters with orthogonal, symplectic and symmetric groups},
  author={Littlewood, Dudley E},
  journal={Canadian Journal of Mathematics},
  volume={10},
  pages={17--32},
  year={1958},
  publisher={Cambridge University Press}
}

@article {gao2022newell,
    AUTHOR = {Gao, Shiliang and Orelowitz, Gidon and Yong, Alexander},
     TITLE = {Newell-{L}ittlewood numbers {II}: extended {H}orn
              inequalities},
   JOURNAL = {Algebr. Comb.},
  FJOURNAL = {Algebraic Combinatorics},
    VOLUME = {5},
      YEAR = {2022},
    NUMBER = {6},
     PAGES = {1287--1297},
      ISSN = {2589-5486},
   MRCLASS = {05E10 (20M99)},
  MRNUMBER = {4529925},
MRREVIEWER = {Himmet\ Can},
       DOI = {10.5802/alco.217},
       URL = {https://doi.org/10.5802/alco.217},
}

@article {gao2025newell,
    AUTHOR = {Gao, Shiliang and Orelowitz, Gidon and Ressayre, Nicolas and
              Yong, Alexander},
     TITLE = {Newell-{L}ittlewood numbers {III}: {E}igencones and
              {GIT}-semigroups},
   JOURNAL = {Compos. Math.},
  FJOURNAL = {Compositio Mathematica},
    VOLUME = {161},
      YEAR = {2025},
    NUMBER = {5},
     PAGES = {1054--1074},
      ISSN = {0010-437X,1570-5846},
   MRCLASS = {22E46 (05E10 14M15)},
  MRNUMBER = {4950555},
MRREVIEWER = {Haian\ He},
       DOI = {10.1112/s0010437x24007759},
       URL = {https://doi.org/10.1112/s0010437x24007759},
}

@article {min2025proof,
    AUTHOR = {Min, Jaewon},
     TITLE = {Proof of the {N}ewell-{L}ittlewood saturation conjecture},
   JOURNAL = {S\'em. Lothar. Combin.},
  FJOURNAL = {S\'eminaire Lotharingien de Combinatoire},
    VOLUME = {93B},
      YEAR = {2025},
     PAGES = {Art. 44, 12},
      ISSN = {1286-4889},
   MRCLASS = {05E10},
  MRNUMBER = {4972054},
}

@article{koike1987young,
  title={Young-diagrammatic methods for the representation theory of the classical groups of type {$B_n$, $C_n$, $D_n$}},
  author={Koike, Kazuhiko and Terada, Itaru},
  journal={Journal of Algebra},
  volume={107},
  number={2},
  pages={466--511},
  year={1987},
  publisher={Elsevier BV}
}

@book{weyl1946classical,
  title={The classical groups: their invariants and representations},
  author={Weyl, Hermann},
  volume={1},
  year={1946},
  publisher={Princeton university press}
}

@article {king2008schur,
    AUTHOR = {King, Ronald C. and Welsh, Trevor A. and van Willigenburg,
              Stephanie J.},
     TITLE = {Schur positivity of skew {S}chur function differences and
              applications to ribbons and {S}chubert classes},
   JOURNAL = {J. Algebraic Combin.},
  FJOURNAL = {Journal of Algebraic Combinatorics. An International Journal},
    VOLUME = {28},
      YEAR = {2008},
    NUMBER = {1},
     PAGES = {139--167},
      ISSN = {0925-9899,1572-9192},
   MRCLASS = {05E05 (05E10)},
  MRNUMBER = {2420783},
MRREVIEWER = {Grant\ Walker},
       DOI = {10.1007/s10801-007-0113-0},
       URL = {https://doi.org/10.1007/s10801-007-0113-0},
}

@article {purbhoo2008on,
    AUTHOR = {Purbhoo, Kevin and van Willigenburg, Stephanie},
     TITLE = {On tensor products of polynomial representations},
   JOURNAL = {Canad. Math. Bull.},
  FJOURNAL = {Canadian Mathematical Bulletin. Bulletin Canadien de
              Math\'ematiques},
    VOLUME = {51},
      YEAR = {2008},
    NUMBER = {4},
     PAGES = {584--592},
      ISSN = {0008-4395,1496-4287},
   MRCLASS = {20G05 (05E05 05E10)},
  MRNUMBER = {2462463},
MRREVIEWER = {Stuart\ Martin},
       DOI = {10.4153/CMB-2008-058-x},
       URL = {https://doi.org/10.4153/CMB-2008-058-x},
}

@article {mcnamara2009positivity,
    AUTHOR = {McNamara, Peter R. W. and van Willigenburg, Stephanie},
     TITLE = {Positivity results on ribbon {S}chur function differences},
   JOURNAL = {European J. Combin.},
  FJOURNAL = {European Journal of Combinatorics},
    VOLUME = {30},
      YEAR = {2009},
    NUMBER = {5},
     PAGES = {1352--1369},
      ISSN = {0195-6698,1095-9971},
   MRCLASS = {05E15},
  MRNUMBER = {2514658},
MRREVIEWER = {Mercedes\ H.\ Rosas},
       DOI = {10.1016/j.ejc.2008.09.026},
       URL = {https://doi.org/10.1016/j.ejc.2008.09.026},
}

@article {ballantine2014schur,
    AUTHOR = {Ballantine, Cristina and Orellana, Rosa},
     TITLE = {Schur-positivity in a square},
   JOURNAL = {Electron. J. Combin.},
  FJOURNAL = {Electronic Journal of Combinatorics},
    VOLUME = {21},
      YEAR = {2014},
    NUMBER = {3},
     PAGES = {Paper 3.46, 36},
      ISSN = {1077-8926},
   MRCLASS = {05E05 (05E10)},
  MRNUMBER = {3262283},
MRREVIEWER = {Eric\ S.\ Egge},
       DOI = {10.37236/3796},
       URL = {https://doi.org/10.37236/3796},
}

@article{chan2026correlation,
  title={Correlation inequalities for {S}chur positivity},
  author={Chan, Swee Hong and Chen, Hong and Pak, Igor and Soskin, Daniel},
  journal={arXiv preprint arXiv:2606.06688},
  year={2026}
}

@inproceedings{lam2024polypositroids,
  title={Polypositroids},
  author={Lam, Thomas and Postnikov, Alexander},
  booktitle={Forum of Mathematics, Sigma},
  volume={12},
  pages={e42},
  year={2024},
  organization={Cambridge University Press}
}

@article{farber2016arrangements,
  title={Arrangements of equal minors in the positive {G}rassmannian},
  author={Farber, Miriam and Postnikov, Alexander},
  journal={Advances in Mathematics},
  volume={300},
  pages={788--834},
  year={2016},
  publisher={Elsevier}
}

@article{gleizer2000littlewood,
  title={Littlewood-{R}ichardson coefficients via {Y}ang-{B}axter equation},
  author={Gleizer, Oleg and Postnikov, Alexander},
  journal={International Mathematics Research Notices},
  volume={2000},
  number={14},
  pages={741--774},
  year={2000},
  publisher={OUP}
}
\bibliographystyle{alpha}

\end{document}